\documentclass[a4paper,12pt,leqno]{amsart}
\usepackage{latexsym}
\usepackage[all]{xy}
\usepackage{amssymb,bm} 
\usepackage{amsmath} 
\usepackage{tikz}
\usetikzlibrary{arrows,cd}  
\usepackage{amsthm}

\usepackage{CJKutf8} 

\def\Z{{\mathbb{Z}}}
\def\K{{\mathbb{K}}}
\def\R{{\mathbb{R}}}
\def\A{{\mathcal{A}}}
\def\B{{\mathcal{B}}}
\def\F{{\mathbb{F}}}
\def\G{{\mathcal{G}}}

\DeclareMathOperator{\codim}{codim}

\DeclareMathOperator{\Der}{Der}

\DeclareMathOperator{\pd}{pd}
\DeclareMathOperator{\depth}{depth}

\DeclareMathOperator{\res}{res}

\DeclareMathOperator{\Hom}{Hom}

\DeclareMathOperator{\Ass}{Ass}  
\DeclareMathOperator{\Min}{Min}  
\DeclareMathOperator{\NFT}{NFT}  
\DeclareMathOperator{\reg}{reg}  
\newcommand{\mideal}{\ensuremath{\mathfrak{m}}} 
\DeclareMathOperator{\Hilb}{Hilb}  
\newcommand{\lk}{\ensuremath{\mathrm{lk}}} 
\DeclareMathOperator{\Tor}{Tor}

\numberwithin{equation}{section}

\newtheorem{theorem}{Theorem}[section]
\newtheorem{prop}[theorem]{Proposition}
\newtheorem{cor}[theorem]{Corollary}
\newtheorem{lemma}[theorem]{Lemma}

\theoremstyle{definition}

\newtheorem{problem}[theorem]{Problem}
\newtheorem{question}[theorem]{Question}

\newtheorem{example}[theorem]{Example}
\theoremstyle{remark}
\newtheorem{rem}[theorem]{Remark}

\title{Residue ideals of hyperplane arrangements}

\author{Takuro Abe}
\address{
Takuro Abe,
Department of Mathematics,
Rikkyo University, 3-34-1 Nishi Ikebukuro, Toshima-ku, 1718501 Tokyo,
Japan.
}
\email{abetaku@rikkyo.ac.jp}

\author{Satoshi Murai}
\address{
Satoshi Murai,
Department of Mathematics
Faculty of Education
Waseda University,
1-6-1 Nishi-Waseda, Shinjuku, Tokyo 169-8050, Japan}
\email{s-murai@waseda.jp}

\date{\today}

\begin{document}

\begin{abstract}
In this paper, we introduce a new idea to study modules of logarithmic differential forms of hyperplane arrangements, which we call residue ideals.
We first establish basic properties of these ideals, including their radicals and primary decompositions, and obtain applications for freeness of restrictions of arrangements.
Then we apply these ideals to the study of modules of logarithmic differential $1$-forms for graphic arrangements. We give an explicit generating set for these modules and find a new connection to cover ideals of graphs studied in combinatorial commutative algebra. As a consequence we establish several new connections between arrangement theory and Stanley--Reisner theory.
\end{abstract}

\maketitle

\section{Introduction}
A linear hyperplane arrangement $\A$ is a finite set of linear hyperplanes in a vector space $V$. The logarithmic vector field $D(\A)$ of $\A$ is one of the central objects of study in the theory of hyperplane arrangements. A general theory and an explicit description of 
the basis for Coxeter arrangements based on invariant theory were introduced by K. Saito in \cite{Sa}, and Terao established fundamental results, especially for free arrangements in \cite{T1}, \cite{T2} and so on. An arrangement $\A$ is free if $D(\A)$ is a free module. This notion is closely related to the classical fact that the invariant ring by the Weyl group action is again a polynomial ring. Since the fundamental invariants can be recovered from the basis for the free Weyl arrangements, and the degrees of a free basis (called the exponents) describe several invariants of arrangements as the exponents of the Weyl group do, logarithmic vector fields and free arrangements can be regarded as a generalization of (co)invariant rings and Weyl groups from geometric and algebraic viewpoints. Recently, by using $D(\A)$, a direct relation between $D(\A)$ and the cohomology groups of regular nilpotent and semisimple Hessenberg varieties, which contains the flag variety, was established in \cite{AHMMS}. So these algebraic objects of arrangements are increasingly important at the intersection of algebra, combinatorics, topology, geometry and representation theory.

The module $\Omega^1(\A)$ of logarithmic differential $1$-forms along a linear hyperplane arrangement $\A$ is the dual of $D(\A)$. So we can investigate the freeness of $\A$ both from $D(\A)$ and $\Omega^1(\A)$. 
However, algebraic properties of $\Omega^1(\A)$ have been studied less extensively than those of the module $D(\A)$ probably because these modules are considered to be difficult to describe. Also, though $D(\A)$ and $\Omega^1(\A)$ are dual to each other, recent developments show that they are very different when $\A$ is not free.

In this paper, to facilitate the study of $\Omega^1(\A)$, we introduce a new ideal $J_\A^H$, which we call the {\bf residue ideal of $\A$ along $H$}\footnote{The terminology ``residue ideal" is motivated by the residue map and should not be confused with the residual ideals introduced in \cite{MN} which is motivated from liaison theory}.
Roughly speaking, the residue ideal captures the behavior of logarithmic $1$-forms around the pole $H \in \A$.
We establish basic properties of these ideals and consider several applications.

Let us quickly recall the module $\Omega^1(\A)$ and define our ideal $J_\A^H$.
Let $\K$ be a field, $V$ an $\ell$-dimensional $\K$-vector space
and $S=\mathrm{sym}(V^*)=\K[x_1,\dots,x_\ell]$,
where $x_1,\dots,x_\ell$ is a basis of the dual space $V^*$.
For a hyperplane $H$ in $V$ we write $\alpha_H$ for a defining linear form of $H$, $Q(\A):=\prod_{H \in \A} \alpha_H$, 
and for a non-zero linear form $\alpha \in S$ we write $H_\alpha$ for the hyperplane defined by the equation $\alpha=0$.
Let
\[
\textstyle \Omega_V^p=\bigwedge^p
\Omega_V^1= \bigoplus_{1 \leq j_1< \cdots <j_p \leq \ell} S (dx_{j_1} \wedge \cdots \wedge dx_{j_p})
\]
be the module of regular $p$-forms on $V$,
where $\Omega_V^1$ is considered to be the free $S$-module with basis $dx_1,\dots,dx_\ell$.
For a linear hyperplane arrangement $\A$,
{\bf the module of logarithmic differential $p$-forms along $\A$} is the submodule of $\frac{1}{Q(\A)}\Omega_V^p$ defined by
\[
\textstyle
\Omega^p(\A)=\{\omega \in \frac 1 {Q(\A)} \Omega_V^p \mid Q(\A) \omega \wedge (d \alpha_H) \in \alpha_H \Omega^{p+1}_V \mbox{ for all } 
H \in \A\}.
\]
Let $\A$ be a linear hyperplane arrangement in $V$
and $H \in \A$.
It is known that we have an exact sequence (\cite{A14})
\begin{align}
    \label{0-1}
0 \longrightarrow
\Omega^1(\A \setminus H)
\lhook\joinrel\longrightarrow \Omega^1(\A)
\stackrel{\mathrm{res}}{\longrightarrow}
\Omega^0(\A^H)=S/(\alpha_H),
\end{align}
where $\A\setminus H$ and $\A^H$ are the deletion of $H$ from $\A$ and the restriction of $\A$ to $H$.
See section 2 for more details on this exact sequence.
We call the image of the map res in \eqref{0-1},
which is an ideal of $S/ (\alpha_H)$,
the {\bf residue ideal $J_\A^H$ of $\A$ along $H$}.
By definition, we have the exact sequence
\begin{align}
    \label{0-2}
0 \longrightarrow
\Omega^1(\A \setminus H)
\lhook\joinrel\longrightarrow \Omega^1(\A)
\stackrel{\mathrm{res}}{\longrightarrow}
J_\A^H \longrightarrow 0,
\end{align}
and it is clear from the exact sequence that we can study $\Omega^1(\A)$ inductively if we know properties of $J_\A^H$.
We now describe the main themes and results of the paper.
\subsection*{Properties of residue ideals}
One of the main goals of this paper is to investigate the basic properties of residue ideals.
Indeed, we prove basic results on residue ideals in \S2, \S3, \S4 and \S5.
Here we introduce two properties of these ideals that we consider particularly important.

Our original motivation of considering the ideal $J_\A^H$ is 
the study of Orlik's conjecture asserting that restrictions of free arrangements are free, which was disproved by Edelman and Reiner in \cite{ER}, 
see Example \ref{ERexample} too. 
However, for almost all known free arrangements, their restrictions are still free. Therefore, it is a natural question to study under which conditions the restriction of a free arrangement is free. To understand what causes the restriction to remain free, we can use the residue ideal.
Indeed, we have the following statement.

\begin{theorem}[{Theorem \ref{thm:3.4}}]
\label{1.2}
Let $\A$ be a free hyperplane arrangement in $V$, $H \in \A$ and $\overline S=S/(\alpha_H)$.
Then
\begin{itemize}
    \item [(1)] $\A\setminus H$ is free $\Leftrightarrow$ $\pd_{\overline S}(J_\A^H)=0$ $\Leftrightarrow$  $J_\A^H=S/(\alpha_H)$.
    \item [(2)] $\A^H$ is free if and only if $\pd_{\overline S} (J_\A^H) \leq 1$.
\end{itemize}
\end{theorem}

Theorem \ref{1.2} enables us to translate the freeness problem of $\A^H$ to the problem on projective dimension of the ideal $J_\A^H$ assuming that $\A$ is free.
We note that, by Terao's restriction theorem \cite{T1}, if both $\A$ and $\A \setminus H$ are free, then $\A^H$ is free.
Our second main result on residue ideals is about their primary decompositions.
The motivation is our computational experience that $J_\A^H$ often has complicated generators
but it often has a rather simple primary decomposition.
For a hyperplane arrangement $\A$ and a subspace $X$ in its intersection lattice $L(\A)$, 
we write
$
\A_X=\{L \in \A \mid X \subseteq L\}$
for the {\bf localization} of $\A$ at $X$. Then by 
summarizing Theorem \ref{thm:4.1}, Lemma \ref{5.1}
and Proposition \ref{5.2}, we obtain the fundamental theorem for primary decomposition of residue ideals.

\begin{theorem}
    \label{thm1.3}
    Let $\A$ be a hyperplane arrangement in $V$,
    $H \in \A$ and
    \[\Xi(\A,H)=\{X \in L(\A^H)\mid J_{\A_X}^H \ne (1) \}.\]
    For a subspace $X$ of $H$,
    let $P_X$ be the defining ideal of $X$ in $\overline S=S/(\alpha_H)$.
    Then
    \begin{itemize}
        \item[(1)] $\sqrt {J_\A^H}= \bigcap_{X \in \Xi(\A,H)} P_X$.
        \item[(2)] the set $\mathrm{Ass}(J_\A^H)$ of associated primes of $J_\A^H$ is contained in $\{P_X \mid X \in \Xi(\A,H)\}$.
        \item[(3)] we have
        $$J_\A^H=\bigcap_{P_X \in \mathrm{Ass}(J_\A^H)} J_{\A_X}^H.$$
        In particular, if $J_\A^H$ has no embedded primes, then the above presentation is a primary decomposition of $J_\A^H$.
        \end{itemize}
\end{theorem}

See \S 4.2 for basics on primary decompositions including the definition of associated primes and embedded primes.
An important point of Theorem \ref{thm1.3} is that, if all the associated primes of $J_\A^H$ are of small height (we believe that this often happens, see Proposition \ref{5.5}), then one can write $J_\A^H$ as an intersection of residue ideals of smaller arrangements.
If $\A$ is free, we can say more. Indeed, for a free arrangement $\A$ we have
\begin{itemize}
    \item $\Xi(\A,H)=\{ X \in L(\A^H) \mid \A_X\setminus H \mbox{ is not free}\}.$
    \item If $P_X$ is a minimal prime of $J_\A^H$ then $J_{\A_X}^H$ is a complete intersection.
    \item If moreover $\A^H$ is free, then 
    \[
\textstyle    J_\A^H=\bigcap_{X \in \Xi(\A,H),\ \dim X=\ell-3} J_{\A_X}^H.
    \]
\end{itemize}
See Corollary \ref{4.3}, Lemma \ref{5.1} and Corollary \ref{cor:newprimarydecompfree}.

\subsection*{Applications to graphic arrangements}
Another goal of this paper is to investigate modules of logarithmic differential $1$-forms for graphic arrangements using residue ideals.

We recall graphic arrangements.
Let $[\ell]=\{1,\ldots,\ell\}$ and let $G=(V(G),E(G))$ be a simple graph with $V(G)=[\ell]$. Then the \textbf{graphic arrangement $\A_G$} of $G$ in $\K^\ell$ is defined by 
$$
\A_G:=\{x_i=x_j\mid \{i,j\} \in E(G)\}.
$$
Clearly there is a one-to-one correspondence between graphs and subarrangements of $\prod_{1 \le i<j\le \ell} (x_i-x_j)=0$. Graphic arrangements have been extensively studied from several aspects. The main topic is how to relate the graph combinatorics with several properties of $\A_G$. For example, it was proved by Stanley that $\A_G$ is free if and only if $G$ is chordal. However, up to now, from the algebraic point of view, existing work on graphic arrangements has focused mainly on logarithmic vector fields; comparatively little is known about logarithmic differential forms. The aim of this part is to establish a clear relationship between $\Omega^1(\A_G)$ and a very well-studied algebro-combinatorial object in combinatorial commutative algebra, a cover ideal.
A  \textbf{cover ideal} $J(G)$ of a graph $G$ is defined by 
$$
J(G):=\bigcap_{\{i,j\}\in E(G)} (x_i,x_j).
$$
Using Theorem \ref{thm1.3} we prove the following result, which shows that  for a graphic arrangement $\A_G$ and $H \in \A_G$ the residue ideal $J_{\A_G}^H$ is radical and gives its prime decomposition.

\begin{theorem}[Theorem \ref{thm:RDforGraph}]
\label{thm:graphic}
Let $\A$ be a graphic arrangement in $\K^\ell$ defined by a graph $G$ and $H=H_{x_i-x_j} \in \A$. 
    Let $W$ be the set of vertices which are adjacent to both $i$ and $j$ in $G$ and let $y_k=x_k-x_i$ for $k \in W$. Then one has
    \[ J_\A^H= \bigcap_{\{u,v\} \subset W, \{u,v\} \not \in E(G)} (y_u,y_v).\]
\end{theorem}

In particular, for a graphic arrangement $\A_G$ and $H \in \A_G$,
the ideal $J_{\A_G}^H$ is nothing but the cover ideal of the complementary graph of the induced subgraph of $G$ on $W$.
This theorem actually induces various new results on $\Omega^1(\A_G)$,
as we explain below.

First, using Theorem \ref{thm:graphic} we find explicit generators of $\Omega^1(\A_G)$.
For a non-empty subset $F \subset [\ell]$, define
\[\gamma_F=\sum_{u \in F} \frac {dx_u} {\prod_{v \in F, v \ne u} (x_v-x_u)}. \]

\begin{theorem}[Theorem \ref{thm:GensOmega}]
\label{thm:graphicGen}    Let $\A_G$ be a graphic arrangement in $\K^\ell$ defined by a graph $G$. 
The set $\{\gamma_F \mid F \ne \varnothing \mbox{ is a clique of $G$}\}$ generates $\Omega^1(\A_G)$.
\end{theorem}

Second, we find a direct relation between $\Omega^1(\A_G)$ and $J(G^c)$.
Recall that all the hyperplanes in $\A_G$ contain the line $x_1=\cdots=x_\ell$. So we can consider $$
\Omega^1_R(\A_G):=\Omega^1(\A_G) \cap \left(\bigoplus_{i=1}^\ell 
\frac{1}{Q(\A_G)}R dx_i\right),
$$
where 
$$
R:=\K[x_i-x_j\mid 1 \le i <j \le \ell].
$$
We note that $\Omega^1_R(\A_G) \otimes_R S \cong \Omega^1(\A_G)$. 

\begin{theorem}[Theorem \ref{thm:6.6}]
\label{omegacoverR}
If $\A_G$ is a graphic arrangement in $\K^\ell$, then
$$
\Omega^1_R(\A_G) \simeq J(G^c)/(x_1\cdots x_\ell).
$$
as $R$-modules.
\end{theorem}

Theorem \ref{omegacoverR} enables us to give a combinatorial formula of the Hilbert series 
for $\Omega^1(\A_G)$, local cohomology modules of $\Omega^1(\A_G)$, and $D(\A_G)$
(Theorem \ref{hilb:omega} and Corollaries \ref{cor:HilbLocalcohomologyOmega} and \ref{cor:HilbD}).
It also shows that $\pd_S(\Omega^1(\A_G))$ coincides with the Castelnuovo--Mumford regularity of the edge ideal $I(G^c)$,
which is the dual object of $J(G^c)$ (Theorem \ref{thm:6.13}). In particular, applying known results on edge ideals, we have the following strong upper bound for the projective dimension of $\Omega^1(\A_G)$.

\begin{theorem}[Corollary \ref{graphpd}]
\label{pdupper}
If $\A_G$ is a graphic arrangement in $\K^\ell$ with $\ell \geq 2$, then
$$
\pd_S \Omega^1(\A_G) \le \frac{\ell}{2}-1.
$$
\end{theorem}

In general, since the module of logarithmic differential forms is reflexive, we know an upper bound of the projective dimension of the module of logarithmic differential $1$-forms:
$$
\pd_S \Omega^1(\A_G) \le \ell-3.
$$
Theorem \ref{pdupper} says that for a graphic arrangement, 
the upper bound of $\pd_S \Omega^1(\A_G)$ is far lower. This gives us the first large class of arrangements whose modules of logarithmic differential $1$-forms cannot attain the maximal possible projective dimensions.
There exist non-graphic arrangements that do not satisfy
the upper bound in Theorem \ref{pdupper}, see Example \ref{exmaxpd}.

By these theorems, 
it turns out that the study of modules of logarithmic differential $1$-forms 
of graphic arrangements is closely related to 
the research of cover ideals and edge ideals of graphs.

This paper is organized as follows. In \S2 we recall the fundamental definitions and results on hyperplane arrangements. Using them, our new ``residue ideal'' is introduced, and several basic results and properties are shown. In \S3 we focus on the residue ideals of free arrangements, and give a characterization of the freeness of the restriction of free arrangements in terms of residue ideals.  \S4 and \S5 are devoted to the study of commutative algebraic properties of residue ideals including the radicality, primary decompositions and unmixedness. In \S6, we recall
fundamental results and definitions on cover ideals, edge ideals and Stanley--Reisner theory. By using them, in \S7, 
we make a direct connection between modules of differential $1$-forms of graphic arrangements and cover ideals of graphs, by proving theorems in the latter half of \S1. We conclude this paper in \S8 by posing several problems and possible applications.
\bigskip

\noindent
\textbf{Acknowledgments}:
We thank Uwe Nagel for letting us know a similarity between Theorem \ref{thm1.3} and primary decompositions of Jacobians. 
We also thank Huy T\`ai H\`a for helpful comments on earlier version of this paper, in particular, for letting us know a connection between strongly chordal graphs and forests in \S 8.
Problem \ref{findBettiFormula} is also inspired from his comments.
The first author
is partially supported by JSPS KAKENHI Grant Numbers JP23K17298 and JP25K24692.
The second author is partially supported by JSPS KAKENHI Grant Numbers JP25K06943.

\section{Preliminaries and basic properties of residue ideals}

In this section,
after recalling some basic facts on hyperplane arrangements and their modules of logarithmic  differential forms,
we prove basic properties of residue ideals.

\subsection{Hyperplane arrangements and residue maps}

Throughout the paper, $V$ is an $\ell$-dimensional $\K$-vector space, and $S=\mathrm{sym}(V^*)=\K[x_1,\dots,x_\ell]$ is the symmetric algebra of the dual space $V^*$.
We consider $S$ as a graded ring by setting $\deg x_i=1$ for all $i$.
A {\bf hyperplane arrangement} in $V$ is a finite collection of hyperplanes in $V$.
A hyperplane arrangement $\A$ is said to be {\bf linear} if every hyperplane in $\A$ contains the origin.
In this paper, we only consider linear hyperplane arrangements and simply call them {\bf arrangements}.
Let $\A$ be an arrangement and $H \in \A$.
The arrangement $\A \setminus H=\A \setminus \{H\}$ is called the {\bf deletion} of $H$ from $\A$ and the arrangement $\A^H=\{H \cap L \mid L \in \A\setminus H\}$ in $H \cong \K^{\ell-1}$ is called the {\bf restriction} of $\A$ to $H$.
The {\bf defining polynomial} $Q(\A)$ of $\A$ is a polynomial 
\[
\textstyle Q(\A)= \prod_{H \in \A} \alpha_H,
\]
and
the {\bf intersection lattice} $L(\A)$ of $\A$ is the poset of all subspaces of $V$ that can be obtained as the intersection of hyperplanes in $\A$ ordered by the reverse inclusion.

For a graded $\K$-algebra $A$ and $f_1,\dots,f_r \in A$,
we write $(f_1,\dots,f_r)$ for the ideal of $A$ generated by $f_1,\dots,f_r$.
Also, $\mideal_A$ will denote the graded maximal ideal of $A$.
From now on,
when we fix an arrangement $\A$ in $V$ and $H \in \A$,
the symbol $\overline S$ denotes the ring $S/(\alpha_H)=\mathrm{sym}(H^*)$,
and for $f \in S$ the symbol $\overline f$ denotes its image to $\overline S$.
If $L$ is a hyperplane in $\A^H$, then $\alpha_L$ is considered to be an element of $\overline S=\mathrm{sym}(H^*)$ and we have $\alpha_L=\overline {\alpha_{L'}}$ for a hyperplane $L'$ of $\A$ with $L=H \cap L'$.

Next, we recall a few basic facts on the logarithmic vector field $D(\A)$ and $\Omega^1(\A)$.
Let $\Der(S)=\bigoplus_{i=1}^\ell S \cdot \partial_{x_i}$ be the module of derivations of $S$.
Algebraically, $\Der(S)$ is the free module generated by $\partial_{x_1},\dots,\partial_{x_\ell}$.
For an arrangement $\A$ in $V$, the {\bf logarithmic vector field} $D(\A)$ of $\A$ is the $S$-module
\[D(\A)=\{\theta \in \Der(S) \mid \theta (\alpha_H) \in (\alpha_H) \mbox{ for any }H \in \A\}.\]
The modules $D(\A)$ and $\Omega^1(\A)$ are $S$-duals of each other by the inner product $\langle \partial_{x_i},dx_j \rangle =\delta_{i,j}$ (see \cite[Theorem 4.75]{OT}).
For any $L \in \A$,
the module $\Omega^1(\A)$ can be decomposed as
\begin{align}
\label{decomp:2}
\Omega^1(\A)=\left( S \cdot \frac {d \alpha_L}{\alpha_L} \right) \oplus \Omega_0^1(\A), 
\end{align}
where
$\Omega_0^1(\A)=\{ \omega \in \Omega^1(\A) \mid \langle \theta_E,\omega \rangle =0\}$
and 
where $\theta_E=\sum_{i=1}^\ell x_i \partial_{x_i}$ is the Euler derivation.
See e.g.\ \cite[\S 3]{Yuz}.
We also note that both $D(\A)$ and $\Omega^1(\A)$ are graded modules
with $\deg (\partial_{x_i})=\deg (dx_i)=0$ for all $i$.

We say that an arrangement $\A$ is \textbf{free} if $D(\A)$ (equivalently $\Omega^1(\A)$) is a free $S$-module.
Let $\A$ be a free arrangement in $V$.
Since $D(\A)$ is of rank $\ell$,
there is a homogeneous $S$-basis $\theta_1,\dots,\theta_\ell$ of $D(\A)$.
The sequence $(d_1,\dots,d_\ell)=(\deg \theta_1,\dots,\deg \theta_\ell)$
is called the {\bf exponents} of $\A$.
Unless $\A = \varnothing$,
a minimal generating set of $D(\A)$ can be chosen so that it contains the Euler derivation $\theta_E$. Hence we may assume $\theta_1=\theta_E$ and $d_1=1$.
We also note that when $\A$ is free with exponents 
$(d_1,d_2,\dots,d_\ell)$,
$\Omega^1(\A)$ is generated by elements of degrees $-d_1,-d_2,\dots,-d_\ell$ since it is the $S$-dual of $D(\A)$.

Now, we explain the exact sequence \eqref{0-1} in the introduction.
Recall that for an arrangement $\A$ in $V$, the module of logarithmic differential $1$-forms along $\A$ is
\[
\Omega^1(\A)=\left\{ \omega \in \frac 1 {Q(\A)} \Omega_V^1 \mid Q(\A) \omega \wedge d\alpha_L \in \alpha_L \Omega^2_V \ \ \mbox{for any }L \in \A
\right\}.
\]
The following fact is immediate from the definition of $\Omega^1(\A).$

\begin{lemma}
\label{2.1}
Let $\A$ be an arrangement in $V$ and $H \in \A$.
If $\alpha_H,\alpha_2,\dots,\alpha_\ell$ form a basis of $V^*$, then for any $\omega =f(d\alpha_H)+f_2(d\alpha_2) + \cdots + f_\ell(d\alpha_\ell) \in \Omega^1(\A)$ one has
$Q(\A) f_2,\dots,Q(\A)f_\ell \in (\alpha_H)$.
\end{lemma}

Let $\A$ be an arrangement in $V$ and $H \in \A$.
The {\bf residue map} $\mathrm{res}^H_\A: \Omega^1(\A) \to \overline S$ of $\A$ along $H$ is defined by\footnote{One may define $\mathrm{res}^H_\A$ by
$\mathrm{res}_\A^H(f (\frac {d\alpha_H}{\alpha_H})+\cdots)=\frac 1 {Q(\A^H)}\overline{Q(\A\setminus H) f}$ since this is a ``residue" map but we use \eqref{2-1} because this formula is convenient when we prove statements.}
\begin{align}
    \label{2-1}
\mathrm{res}^H_\A \big(f(d\alpha_H)+f_2(d\alpha_2) + \cdots + f_\ell(d\alpha_\ell)\big)= \frac 1 {Q(\A^H)} \overline {Q(\A) f},
\end{align}
where $\alpha_H,\alpha_2,\dots,\alpha_\ell$ is a basis of $V^*$ and $f,f_2,\dots,f_\ell \in \frac 1{Q(\A)} S$.
Lemma \ref{2.1} guarantees that the RHS of \eqref{2-1} is independent of the choice of $\alpha_2,\dots,\alpha_\ell$.
It is not trivial that RHS of \eqref{2-1} is a polynomial, but the following fact is known.

\begin{lemma}[{\cite[Theorem 5.3]{A14}}]
    \label{2.2}
    The RHS of \eqref{2-1} is contained in $\overline S$ and we have the exact sequence
\[
0 \longrightarrow
\Omega^1(\A \setminus H)
\lhook\joinrel\longrightarrow \Omega^1(\A)
\stackrel{\mathrm{res}^H_\A}{\longrightarrow}
\overline S.
\]   
\end{lemma}

\begin{proof}
The lemma is a special case of \cite[Theorem 5.3]{A14}, but we include the proof for the sake of completeness.
The statement is obvious when $\ell=1$ so we assume $\ell\geq 2$
(note that if $\A = \varnothing$ then we consider that $\Omega^1(\A)=\Omega_V^1$).
Let $\omega \in \Omega^1(\A)$ and $L \in \A^H$.
To prove the first assertion,
what we must prove is that $Q(\A^H)\mathrm{res}_\A^H(\omega) \in \overline S$ is divisible by $\alpha_L$.
Take $L' \in \A$ such that $L'\cap H=L$
and take a basis $\alpha_H,\alpha_{L'},\alpha_3,\dots,\alpha_\ell$ of $V^*$.
If we write $\omega = f(d\alpha_H)+f_2(d\alpha_{L'})+f_3(d\alpha_3)+\cdots +f_\ell(d\alpha_\ell)$ then $\alpha_{L'}$ must divide $Q(\A)f$ by Lemma \ref{2.1}
and hence
\[
Q(\A^H)\mathrm{res}_\A^H(\omega)= \overline {Q(\A)f}
\]
is divisible by $\alpha_L=\overline {\alpha_{L'}}$ as desired.

Next, we prove the exactness.
Let $\omega \in \Omega^1(\A)$.
Observe that by the definition of $\Omega^1(\A)$ we have
\[
\omega \in \Omega^1(\A \setminus H) \ \Leftrightarrow\  Q(\A\setminus H) \omega \in \Omega_V^1 \ \Leftrightarrow \ Q(\A) \omega \in \alpha_H \Omega_V^1.
\]
On the other hand, by the definition of $\mathrm{res}_\A^H$ and Lemma \ref{2.1},
we also have
\[
\res_\A^H(\omega)=0 \ \Leftrightarrow \ Q(\A) \omega \in \alpha_H \Omega_V^1.
\]
These prove the desired equality $\Omega^1(\A\setminus H) = \mathrm{ker}(\mathrm{res}^H_\A)$.
\end{proof}

From now on, to simplify the notation, we write $\mathrm{res}^H=\mathrm{res}^H_\A$ when the arrangement $\A$ is clear.

\begin{example}
    \label{ex:2.3}
    Let $\A$ be the arrangement in $\R^4$ with
    \[Q(\A)=(x_1-x_3)(x_1-x_2)(x_2-x_3)(x_3-x_4)(x_1-x_4)\]
    and $H=H_{x_1-x_3} \in \A$.
    In this case, $Q(\A^H)=(x_1-x_2)(x_1-x_4) \in \overline S=S/(x_1-x_3)$
    and $\Omega^1(\A)$ is the free module generated by the following four elements
    (these generators are computed using Theorem \ref{thm:graphicGen} (see also Lemma \ref{lem:gensOmega})):
    \begin{align*}
        \omega_1&= dx_1,\\
        \omega_2&= \frac 1 {x_1-x_2} (dx_2-dx_1),\\
        \omega_3&= \frac 1 {(x_1-x_2)(x_3-x_2)} (dx_2-dx_1)+\frac 1 {(x_1-x_3)(x_2-x_3)} (dx_3-dx_1),\\
 \omega_4&= \frac 1 {(x_1-x_3)(x_4-x_3)} (dx_3-dx_1)+\frac 1 {(x_1-x_4)(x_3-x_4)} (dx_4-dx_1).
    \end{align*}
The residue ideal of $\A$ along $H$ is given by
{\small
\begin{align*}
J_\A^H&=\big(\mathrm{res}^H(\omega_1),\mathrm{res}^H(\omega_2),\mathrm{res}^H(\omega_3),\mathrm{res}^H(\omega_4)\big)\\
    &=\left(0,0,\frac 1 {Q(\A^H)} \!\left( \overline {Q(\A) \frac 1 {(x_1\!-\!x_3)(x_2\!-\!x_3)}}
\right),
\frac 1 {Q(\A^H)}\!\left( \overline {Q(\A) \frac 1 {(x_1\!-\!x_3)(x_4\!-\!x_3)}}
\right)
\right)\\
&=(x_4-x_1,x_2-x_1),
\end{align*}
}

\noindent
where we use the fact that $x_1=x_3$ in $S/(x_1-x_3)$ for the last equality.
\end{example}

\subsection{Basic properties of residue ideals}
In the rest of this section,
we discuss basic properties of residue ideals.
First, we consider relations between $J_\A^H$ and $J_{\A \setminus H'}^H$.
Let $\A$ be an arrangement and $H \in \A$.
For $L \in \A^H$, we call the number
\[ \mathrm{mult}^\A(L)=\# \{ H' \in \A \mid H' \cap H=L\}\]
the \textbf{multiplicity} of $L$ with respect to \ $\A$,
where $\#A$ denotes the cardinality of a finite set $A$.
The following fact easily follows from the definition of residue maps.

\begin{lemma}
\label{lem:deletionIAH}
Let $\A$ be an arrangement and $H,H' \in \A$.
If $H\cap H' \in \A^H$ has multiplicity $1$ with respect to\ $\A$, then $J_{\A \setminus H'}^H \subset J_\A^H$.
\end{lemma}

\begin{proof}
Let $L=H \cap H'$.
Since $L$ has multiplicity $1$ with respect to \ $\A$, we have
\[
\frac {\overline{Q(\A)}} {Q(\A^H)} = \frac 
{\overline{Q(\A\setminus H') \cdot \alpha_{H'}} } {Q((\A\setminus H')^H) \cdot \alpha_L}
=
\frac 
{\overline{Q(\A\setminus H')} } {Q((\A\setminus H')^H)},
\]
where we use $\overline {\alpha_{H'}}=\alpha_L$
for the second equality.
Then the statement follows from the definition of the residue map together with the fact that $\Omega^1(\A\setminus H') \subset \Omega^1(\A)$.\end{proof}

\begin{rem}
One has $\alpha_{H'} J_{\A\setminus H'}^H \subset J_\A^H$
even if the multiplicity of $H \cap H'$ is not $1$.
However, we will not use this fact in the paper.
\end{rem}

Next, we show that residue ideals are essentially unchanged under essentialization.
For an arrangement $\A$, let $X_\A=\cap_{H \in \A} H$ denote the maximal element of $L(\A)$.
We say that an arrangement $\A$ is {\bf essential} if $X_\A=\{0\}$.
For an arrangement $\A$ in $V$ and a subspace $X$ of $X_\A$,
we write $\A^{V/X}=\{H/X \mid H \in \A\}$,
which is an arrangement in $V/X$.
The arrangement $\A^e=\A^{V/X_\A}$, which is an essential arrangement, is called the {\bf essentialization} of $\A$.

\begin{lemma}
\label{lem:2.5}
Let $\A$ be an arrangement in $V$ and $X$ a subspace of $X_\A$.
Let $\alpha_1,\dots,\alpha_\ell$ be a basis of $V^*$ such that $\alpha_1,\dots,\alpha_m$ is a basis of $(V/X)^*$. Then 
\begin{itemize}
    \item[(1)] $\Omega^1(\A^{V /X})=\Omega^1(\A) \cap \big( \bigoplus_{i=1}^m \big(\mathrm{sym}((V/X)^*)\cdot \frac 1 {Q(\A)} d \alpha_i \big) \big).$
    \item[(2)]
    $\Omega^1(\A)$ is generated by elements of $\Omega^1(\A^{V/X})$ and $d\alpha_{m+1},\dots,d\alpha_\ell$.
\end{itemize}
\end{lemma}

\begin{proof}
Observe that by the definition of $\Omega^1(\A)$ one has
\[Q(\A) \Omega^1(\A)= \bigcap_{H \in \A} Q(\{H\}) \Omega^1(\{H\}). \]
The desired statements follow from the fact that $Q(\{H\})\Omega^1(\{H\})$ is generated by $d(\alpha_H)$ and $\alpha_H(d\alpha_1),\dots,\alpha_H(d\alpha_\ell)$
together with the observation that $\alpha_H$ is contained in $\mathrm{span}_\F\{\alpha_1,\dots,\alpha_m\}$ for any $H \in \A$.
\end{proof}

The next corollary follows immediately from the above lemma and the definition of residue ideals.

\begin{cor}
\label{2.6}
Let $\A$ be an arrangement in $V$ and $X$ a subspace of $X_\A$. One has $J_\A^{H}=J_{\A^{V/X}}^{H/X} \! \cdot \overline S$. In particular, if we write $\tilde H=H/X_\A$ then we have
\[ J_\A^H = J_{\A^e}^{\tilde H} \!\cdot  \overline S.\]
\end{cor}

Note that in the above corollary,
$J_{\A^{V/X}}^{H/X}$ is an ideal of $\mathrm{sym}((H/X)^*)$ which is a subring of $\mathrm{sym} (H^*)=\overline S$.

\begin{example}
    \label{ex:2.8}
    Consider the arrangement $\A$ in Example \eqref{ex:2.3}.
    In this case $X_\A=\{(x_1,\dots,x_4)\in \mathbb R^4 \mid x_1=x_2=x_3=x_4\}$ and $\Omega^1(\A^e)$ is generated by $\omega_2,\omega_3$ and $\omega_4$. Let $R=\mathrm{sym}((\mathbb R^4/X_\A)^*)=\R[x_2-x_1,x_3-x_1,x_4-x_1]$. For $H=H_{x_1-x_3}$, we have
    \[
    J_{\A^e}^{\tilde H}=(\res^H(\omega_2),\res^H(\omega_3),\res^H(\omega_4))=(x_4-x_1,x_2-x_1) \subset R/(x_3-x_1).
    \]
\end{example}

For the arrangement $\A$ in Example \ref{ex:2.3},
we see that $\Omega^1(\A)$ is generated by four elements $\omega_1,\dots,\omega_4$ but residue ideal $J_\A^H$ is generated by just two elements.
Corollary \ref{2.6} explains why $\omega_1$ does not contribute to $J_\A^H$.
The next lemma explains why one more generator disappears by the map $\res^H$.

\begin{lemma}
    \label{lem:3.3}
Let $\A$ be an arrangement in $V$ with $\# \A \geq 2$ and $H \in \A$.
Then
\[J_\A^H=\{ \mathrm{res}^H(\omega) \mid \omega \in \Omega_0^1(\A)\}.\]
\end{lemma}

\begin{proof}
Let $L \in \A$ with $L \ne H$.
Since $\mathrm{res}^H(d\alpha_L/\alpha_L)=0$,
the statement follows from the decomposition \eqref{decomp:2}.
\end{proof}

\begin{example}
Consider the arrangement $\A$ in Example \ref{ex:2.3}.
Let $R=\R[x_2-x_1,x_3-x_1,x_4-x_1]$.
Then $\Omega^1(\A^e)=(R \cdot \omega_2) \oplus \Omega_0^1(\A^e)$
with $\Omega^1_0(\A^e)=R \cdot \omega_3+R \cdot \omega_4$,
and Lemma \ref{lem:3.3} guarantees
\[
J_\A^H=\big(\mathrm{res}^H(\omega_3),\mathrm{res}^H(\omega_4)\big).\]
\end{example}

For free arrangements we can get more detailed information on generating sets of residue ideals.
Recall that if $\A$ is a free arrangement with exponents 
$(d_1=1,d_2,\dots,d_\ell)$,
then $\Omega^1(\A)$ is generated by elements of degrees $-d_1=-1,-d_2,\dots,-d_\ell$.
Applying Lemma \ref{lem:3.3} to free arrangements,
we get the following statement.

\begin{lemma}
    \label{lem:3.3sononi}
With the same notation as in Lemma \ref{lem:3.3},
if $\A$ is free with exponents $(d_1,d_2,\dots,d_\ell)$ 
with $d_1=1$, then 
$J_\A^H=(f_2,\dots,f_\ell)$ for some polynomials $f_2,\dots,f_\ell$ with $\deg f_i=\#\A -\# \A^H -d_i$ for $i=2,3,\dots,\ell$.
\end{lemma}

Finally, we show that residue ideals behave nicely with respect to localizations.

\begin{lemma}
\label{2.8}
Let $\A$ be an arrangement in $V$ and $H \in \A$.
For every $X \in L(\A^H)$, one has
\[J_\A^H \subset J_{\A_X}^H.\]
\end{lemma}

\begin{proof}
Let $f \in J_\A^H$.
We claim $f \in J_{\A_X}^H$.
Let $\omega \in \Omega^1(\A)$ with $\mathrm{res}^H_\A(\omega)=f$.
Since $Q(\A) \Omega^1(\A) \subset Q(\A_X) \Omega^1(\A_X)$,
we have $\frac {Q(\A)} {Q(\A_X)} \omega \in \Omega^1(\A_X)$.
Thus we have
\begin{align}
    \label{eq2-2}
\mathrm{res}^H_{\A_X}
\left( \frac {Q(\A)} {Q(\A_X)} \omega \right)
 =
\frac {\overline{Q(\A_X) }} {\overline{Q(\A)}} \frac {Q(\A^H)} {Q\big((\A_X)^H\big)} \mathrm{res}_\A^H\left( \frac {Q(\A)} {Q(\A_X)} \omega \right)
= \frac {Q(\A^H)} {Q\big((\A_X)^H\big)} f
\end{align}
is contained in $J_{\A_X}^H$.
On the other hand, by Corollary \ref{2.6},
$J_{\A_X}^H$ is generated by polynomials in $\mathrm{sym}((H/X)^*)\subset \overline S$.
Thus, if a linear form $\alpha_L \in \overline S$ is not in $(H/X)^*$, equivalently if $X \not \subset L$, then $\alpha_L$ is a non-zero divisor of $\overline S/J_{\A_X}^H$.
This says that
$$\frac {Q(\A^H)} {Q\big((\A_X)^H\big)} =\prod_{L \in \A^H,\ X \not \subset L} \alpha_L$$
is a non-zero divisor  of $S/J_{\A_X}^H$.
Then, since we have already proved that $\frac {Q(\A^H)} {Q((\A_X)^H)}f \in J_{\A_X}^H$,
we have $f \in J_{\A_X}^H$ as desired.
\end{proof}

For an arrangement $\A$ and $H \in \A$, we define
\[
\textstyle
\A_{\mathrm{core}}^H=\{L \in \A^H \mid \mathrm{mult}^\A(L) \geq 2\}
\ \mbox{ and }\ 
X_{\mathrm{core}}^{(\A,H)}=\bigcap_{L \in \A_{\mathrm{core}}^H}L 
.\]
The next proposition, which follows from Lemmas \ref{lem:deletionIAH} and \ref{2.8},
is often convenient to compute residue ideals efficiently.

\begin{prop}
\label{prop:2.10}
Let $\A$ be an arrangement, $H \in \A$ and $X=X_{\mathrm{core}}^{(\A,H)}$. Then
\[ J_\A^H=J_{\A_X}^H.\]
\end{prop}

\begin{proof}
Since $\mathrm{mult}^\A(L \cap H)=1$ for any $L \in \A \setminus \A_X$ by the definition of $X$,
we have $J_{\A_X}^H \subset J_\A^H$ by Lemma \ref{lem:deletionIAH}. Since Lemma \ref{2.8} gives the reverse inclusion,
we have $J_{\A_X}^H=J_\A^H$ as desired.
\end{proof}

\begin{example}
    Let $\B$ be the arrangement in $\R^6$ with
    \[Q(\B)=(x_1\!-\!x_3)(x_1\!-\!x_2)(x_2\!-\!x_3)(x_3\!-\!x_4)(x_1\!-\!x_4)(x_1\!-\!x_5)(x_2\!-\!x_5)(x_3\!-\!x_6),\]
    $H=H_{x_1-x_3}$ and $\overline S=\R[x_1,\dots,x_6]/(x_1-x_3)$.
    Then 
    \[\B_{\mathrm{core}}^H=\{H_{\overline {x_1-x_2}}=H\cap H_{x_1-x_2},H_{\overline {x_1-x_4}}=H\cap H_{x_1-x_4}\}\] and
    \[X_{\mathrm{core}}^{(\B,H)}=\{(x_1,\dots,x_6) \in \mathbb R^6 \mid x_1=x_2=x_3=x_4\}.\]
    Then, since $\B_{X_{\mathrm{core}}^{(\B,H)}}$ has the same defining polynomial as the arrangement $\A$ in Example \ref{ex:2.3}, the ideal $J_\B^H$ has the same generating set as $J_\A^H$ by Proposition \ref{prop:2.10}. Hence we conclude 
    \[J_\B^H=(x_4-x_1,x_2-x_1).\]
\end{example}

For a commutative ring $A$ and its prime ideal $P$,
we denote by $M_P$ the localization of an $A$-module $M$ at the prime ideal $P$.
Let $\A$ be an arrangement in $V$, $H \in \A$ and $\overline S=S/(\alpha_H)$.
For a prime ideal $P$ in $\overline S$,
we write $\tilde P$ for the prime ideal in $S$ with $S/\tilde P=\overline S/P$
and write $\A_P=\{ L \in \A \mid \alpha_L \in \tilde P\}.$
Note that $\A_P$ is nothing but the localization of the arrangement $\A$ with respect to the subspace $X_{\A_P}=\bigcap_{H' \in \A_P} H' \in L(\A^H)$.
Also, for an ideal $J$ of $\overline S$,
the localization $J_P$ at $P$ as an $\overline S$-module
and the localization $J_{\tilde P}$ as an $S$-module
coincide naturally. Let us recall the following.

\begin{lemma}[{\cite[Example 4.122]{OT}}]
\label{localfunctor}
The functor associating an arrangement $\A$ to the $S$-module $\Omega^1(\A)$ is local, i.e., for $X \in L(\A)$ and a prime ideal $P \subset S$ with $\{H \in \A \mid \alpha_H \in P \}= \A_X$,
it holds that 
$$
\Omega^1(\A)_P=\Omega^1(\A_X)_P.
$$
\end{lemma}

\begin{lemma}
    \label{2.12}
    With the same notation as above,
    for any prime ideal $P$ in $\overline S$
    one has 
    \[(J_\A^H)_P=(J_{\A_P}^H)_P.\]
\end{lemma}

\begin{proof}
This follows immediately from the exact sequence
\[
0 \longrightarrow \Omega^1(\A\setminus H)
\lhook\joinrel\longrightarrow \Omega^1(\A) \stackrel {\mathrm{res}^H} \longrightarrow J_\A^H \longrightarrow 0
\]
together with Lemma \ref{localfunctor}.
\end{proof}

\section{Residue ideals of free arrangements}

In this section, we study properties of $J_\A^H$ when $\A$ is a free arrangement.
To do this, we need some tools on logarithmic vector fields and modules of logarithmic differential forms of higher orders.

\subsection{Vector fields and differential forms of higher orders}

Recall that $V$ is an $\ell$-dimensional $\K$-vector space, $x_1,\dots,x_\ell$ a basis of $V^*$
and $S=\K[x_1,\dots,x_\ell]$.
Let
\[
\Der ^p(S)=\bigwedge^p
\Der(S)= \bigoplus_{1 \leq j_1< \cdots <j_p \leq \ell} S (\partial_{x_{j_1}} \wedge \cdots \wedge \partial_{x_{j_p}}).
\]
For an arrangement $\A$ in $V$,
its {\bf logarithmic vector field of order $p$} is the module
\begin{align*}
D^p(\A)
\!=\!\{\theta \!\in\! \mathrm{Der}^p(S) \mid\theta(\alpha_H,f_2,\dots,f_p) \!\in\! (\alpha_H) \mbox{ for any }H \!\in\! \A \mbox{ and }f_2,\dots,f_p \!\in\! S\},
\end{align*}
where $\theta(g_1,\dots,g_p)$ denotes the substitution of $g_1,\dots,g_p \in S$ to $\theta$.
Alternatively, if we write
$d_\alpha$ with $\alpha \in V^*$ for the $S$-linear map from $\Der^\bullet (S)$ to $\Der^{\bullet -1}(S)$ induced by
\[
d_\alpha(\partial_{x_{j_1}}\wedge \cdots \wedge \partial_{x_{j_p}})
= \sum_{k=1}^p ((-1)^{k-1}(\partial_{x_{j_k}} \alpha_H)) (\partial_{x_{j_1}}\wedge \cdots \wedge \widehat {\partial_{x_{j_k}}}\wedge \cdots \wedge \partial_{x_{j_p}}) ),
\]
then $D^p(\A)$ can be written as
\[
D^p(\A)=\{\theta \in \Der^p(S) \mid d_{\alpha_H} (\theta) \in \alpha_H \Der^{p-1}(S) \ \mbox{ for all } H \in \A\}.
\]
Below we list some properties of $D^p(\A)$ and $\Omega^p(\A)$ which we need.
The first lemma is immediate from the definition, so we omit a reference.

\begin{lemma}
\label{P1}
For an arrangement $\A$ in $V$,
one has
    $\Omega^\ell(\A)=\frac 1 {Q(\A)} \Omega_V^\ell$  and $D^\ell(\A)=Q(\A) \Der^\ell(S)$.
\end{lemma}

For $F=\{j_1,\dots,j_p\} \subset [\ell]$ with $j_1< \cdots < j_p$,
we write $\partial_F= \partial_{x_{j_1}}\wedge \cdots \wedge \partial_{x_{j_p}}$ and $d_F=dx_{j_1}\wedge \cdots \wedge d{x_{j_p}}$. 
Also, we define $\mathrm{sign}(F) \in \{ \pm 1\}$ by the equation $\mathrm{sign}(F) \partial_F \wedge \partial _{[\ell] \setminus F}=\partial_{x_1} \wedge \cdots \wedge \partial_{x_\ell}$.

\begin{lemma}[{\cite[Remark 2.3]{AY}}]
\label{P2}
For an arrangement $\A$ in $V$, we have an isomorphism
\[Q(\A) \Omega^p(\A) \cong D^{\ell-p}(\A)\]
which sends each $d_F$ to $ \mathrm{sign}(F) \partial_{[\ell]\setminus F}$.
\end{lemma}

\begin{lemma}[{\cite[\S 1]{Sa}}]
\label{P3}
For an arrangement $\A$ in $V$, 
one has 
\[
\mbox{
$\Hom_S(D^p(\A),S) \cong \Omega^p(\A)$
\ and \ 
$\Hom_S(\Omega^p(\A),S) \cong D^p(\A)$}
\]
for all $p$.
In particular, the modules $D^p(\A)$ and $\Omega^p(\A)$ are reflexive.
\end{lemma}

\begin{lemma}[{\cite[Proposition 4.81]{OT}}]
\label{P5}
If $\A$ is a free arrangement,
then $D^p(\A) = \bigwedge^p D(\A)$
and $\Omega^p(\A)=\bigwedge^p \Omega^1(\A)$ for all $p$.
\end{lemma}

\begin{lemma}[{\cite[Theorem 4.37]{OT}}]
If $\A$ is free, then $\A_X$ is free for $X \in L(\A)$.
\label{localizationfree}
\end{lemma}

For an arrangement $\A$ and $H \in \A$,
we have the following two types of exact sequences \eqref{3-1} (see \cite[page 1195]{DSSWW}) and \eqref{3-2} (see \cite[Proposition 2.4]{AD}),
which we call the Euler sequence and the dual Euler sequence, respectively:
\begin{align}
    \label{3-1}
    0 \longrightarrow D^p(\A\setminus H) \stackrel {\cdot \alpha_H} \longrightarrow D^p(\A) \stackrel {\rho_p^H} \longrightarrow D^p(\A^H),\\
    \label{3-2}
    0 \longrightarrow \Omega^p(\A) \stackrel {\cdot \alpha_H}  \longrightarrow \Omega^p(\A \setminus H) \stackrel {j_H^p}  \longrightarrow \Omega^p(\A^H),
\end{align}
where $\rho^H_p(\sum_{F \subset [\ell], |F|=p} f_F \partial_F)=\sum_{F \subset [\ell], |F|=p} \overline {f_F} \partial_F$ is induced by the projection $S \to \overline S=S/(\alpha_H)$
(we omit the description of the map $j_H^p$ since we do not need it).
About the surjectivity of the rightmost maps in these sequences,
the following result is known.
See \cite[Theorem 1.13]{A9} and \cite[Theorem 3.3 and Proposition 3.4]{A14}.

\begin{lemma}
\label{P4}
The map $\rho^H_p$ in the Euler sequence \eqref{3-1} is surjective when $\A\setminus H$ is free and the map $j_H^p$ in the dual Euler sequence \eqref{3-2} is surjective when $\A$ is free.
\end{lemma}

\subsection{Determinant formula for residue ideals.}
Next, we show that when $\A$ is free then the residue ideal can be computed from $D(\A)=D^1(\A)$.

Let $\A$ be an arrangement in $V$.
Fix a basis $\alpha_1,\dots,\alpha_\ell$ of $V^*$.
For $\theta_1,\dots,\theta_\ell \in \Der(S)$,
where $\theta_i=\sum_{j=1}^\ell f_{ij} \partial_{\alpha_j}$,
its {\bf Saito Matrix} $M(\theta_1,\dots,\theta_\ell)$ with respect to a basis $\alpha_1,\dots,\alpha_\ell$ is the $(\ell \times \ell)$-matrix whose $(i,j)$th entry is $f_{ji}$.
It is well-known that if $\theta_1,\dots,\theta_\ell \in D(\A)$,
then the determinant of $M(\theta_1,\dots,\theta_\ell)$ 
is divisible by the defining polynomial $Q(\A)$ of $\A$ (see e.g.\ \cite[Proposition 4.12]{OT}). Also, for the later use, let us recall the most fundamental result for the freeness.

\begin{theorem}[{Saito's criterion \cite{Sa}. See also \cite[Theorem 4.19]{OT}}]
\label{Saitocriterion}
Let $\theta_1,\ldots,\theta_\ell \in D(\A)$ be homogeneous derivations. Then they form a basis for $D(\A)$ (so $\A$ is free) if and only if 
\begin{itemize}
\item[(1)]
$\theta_1,\ldots,\theta_\ell$ are $S$-independent, and 
\item[(2)]
$\sum_{i=1}^\ell \deg \theta_i=|\A|$.
\end{itemize}
In particular, $\A$ is free if and only if $\det M(\theta_1,\ldots,\theta_\ell)=c Q(\A)$ for some $c \in \K \setminus \{0\}$.

The same holds true for the module of logarithmic differential $1$-forms, i.e., 
if $\omega_1,\ldots,\omega_\ell \in \Omega^1(\A)$ are homogeneous forms, then they form a basis for $\Omega^1(\A)$ (so $\A$ is free) if and only if 
\begin{itemize}
\item[(1)]
$\omega_1,\ldots,\omega_\ell$ are $S$-independent, and 
\item[(2)]
$\sum_{i=1}^\ell \deg \omega_i=-|\A|$.
\end{itemize}

\end{theorem}

\begin{lemma}
\label{lem:detformula}
Let $\A$ be a free arrangement in $V$ and
$H \in \A$.
For $\theta=\sum_{i=1}^\ell f_i \partial_{x_i} \in D(\A)$, we write $\overline \theta= \rho_1^H(\theta)=\sum_{i=1}^\ell \overline f_i \partial_{x_i} \in D(\A^H)$.
Then
\[J_\A^H= \left\{ \frac 1 {Q(\A^H)} \det M(\overline \theta_1,\dots,\overline \theta_{\ell-1}) \mid \theta_1,\dots,\theta_{\ell-1} \in D(\A) \right\},\]
where Saito matrices are taken with respect to a fixed basis of $H^*$.
\end{lemma}

\begin{proof}
Recall that $J_\A^H$ is the image of the residue map $\mathrm{res}^H$ in the exact sequence
\[
0 \longrightarrow \Omega^1(\A \setminus H)
\lhook\joinrel\longrightarrow
\Omega^1(\A) 
\stackrel{\mathrm{res}^H} \longrightarrow \overline S=S/(\alpha_H)=\Omega^0(\A^H).
\]
The above exact sequence can be re-written as the exact sequence
\[
0 \longrightarrow  Q(\A\setminus H) \cdot  \Omega^1(\A \setminus H)
\stackrel{\times \alpha_H}{\longrightarrow}
Q(\A) \Omega^1(\A) 
\stackrel{\psi} \longrightarrow Q(\A^H) \overline S,
\]
where $\psi(\omega)=Q(\A^H) \mathrm{res}^H(  {Q(\A)}^{-1} \omega )$
for $\omega \in Q(\A) \Omega^1(\A)$.
Let $y_1,\dots,y_{\ell-1}$ be a basis of $H^*$.
By the isomorphism in Lemma \ref{P2},
the above exact sequence agrees with the exact sequence
\[
0 \longrightarrow D^{\ell-1}(\A \setminus H)
\stackrel {\cdot \alpha_H} \longrightarrow
D^{\ell-1}(\A) 
\stackrel{\rho_{\ell-1}^H} \longrightarrow 
D^{\ell-1}(\A^H)=Q(\A^H) \overline S \cdot (\partial_{y_1} \wedge \cdots \wedge \partial_{y_{\ell-1}}).
\]
Hence
\[
J_\A^H\cdot (\partial_{y_1} \wedge \cdots \wedge \partial_{y_{\ell-1}})
=  \frac 1 {Q(\A^H)} \mathrm{Im}\ \! \psi \cdot (\partial_{y_1} \wedge \cdots \wedge \partial_{y_{\ell-1}}) = \frac 1{Q(\A^H)} \mathrm{Im}\ \! \rho_{\ell-1}^H.
\]
Since Lemma \ref{P5} says $D^{\ell-1}(\A)=\bigwedge ^{\ell-1} D(\A)$ by the freeness of $\A$,
we have
\begin{align}
    \label{eq:3.1}
J_\A^H\cdot (\partial_{y_1} \wedge \cdots \wedge \partial_{y_{\ell-1}})
=\frac 1 {Q(\A^H)} \{ \overline \theta_1 \wedge \cdots \wedge \overline \theta_{\ell-1} \mid \theta_1,\dots,\theta_{\ell-1} \in D(\A)\}.
\end{align}
On the other hand,
for $\theta_1,\dots,\theta_{\ell-1} \in D(\A)$,
we have
\begin{align}
    \label{eq:3.2}
    \overline \theta_1 \wedge \cdots \wedge \overline \theta_{\ell-1} = \det M(\overline \theta_1,\dots,\overline \theta_{\ell-1}) \partial_{y_1}\wedge \cdots \wedge \partial_{y_{\ell-1}},
\end{align}
where Saito matrices are taken with respect to $y_1,\dots,y_{\ell-1}$.
Then the equations \eqref{eq:3.1} and \eqref{eq:3.2} prove the desired statement.
\end{proof}

\begin{example}
    \label{ex:3-2}
    Let $\A$ be an arrangement in $\mathbb R^3$ with $Q(\A)=x(y^2-x^2)(z^2-x^2)$.
    Using Theorem \ref{Saitocriterion}, one can easily check that $D(\A)$ is a free $S$-module generated by
    \begin{align*}
        \theta_1 & = x\partial_x + y \partial_y + z \partial_z,\\
        \theta_2 & = (y^2-x^2)\partial_y,\\
        \theta_3 & = (z^2-x^2)\partial_z.
    \end{align*}
    Let $H=H_x$ and $S=\R[x,y,z]$. Then $Q(\A^H)=yz \in S/(x)$ and Lemma \ref{lem:detformula} says
    \[
    J_\A^H=\frac 1 {yz} \cdot I_2 \begin{pmatrix}
        y & y^2-x^2 & 0 \\ z & 0 & z^2-x^2 & 
    \end{pmatrix},
    \]
where $I_k(M)$ denotes the ideal generated by $k$-minors of a matrix $M$.
Since $y^2-x^2=y^2$ and $z^2-x^2=z^2$ in $S/(x)$ we have
\[
\textstyle
J_\A^H=\frac 1 {yz} (y^2z^2,yz^2,-zy^2)=(y,z) \subset S/(x).
\]
Similarly, if we set $L=H_{y-x} \in \A$, then we have $Q(\A^H)=x(z+x)(z-x) \in S/(y-x)$ and
\[
J_\A^L= \frac 1 {x(z+x)(z-x)} \cdot I_2 \begin{pmatrix}
    y=x & y^2-x^2=0 & 0 \\ z & 0 & z^2-x^2
\end{pmatrix}
=S/(y-x).\]
\end{example}

\subsection{Connection to the freeness of $\A \setminus H$ and $\A^H$}

As we explained in the introduction,
if $\A$ is a free arrangement, 
then the projective dimension of the ideal $J_\A^H$ is closely related to the freeness of $\A\setminus H$ and $\A^H$.

\begin{theorem}
    \label{thm:3.4}
    Let $\A$ be a free arrangement in $V$ and $H \in \A$. Then
    \begin{itemize}
        \item [(1)] 
        $\A \setminus H$ is free $\Leftrightarrow$ $\mathrm{pd}_{\overline S} (J_\A^H)=0$ $\Leftrightarrow$ $J_\A^H=\overline S$. 
        \item[(2)] $\A^H$ is free $\Leftrightarrow$ $\pd_{\overline S}(J_\A^H) \leq 1$.
    \end{itemize}
\end{theorem}

Before the proof, we note the following fact.

\begin{lemma} 
\label{lem:3.5}
Let $\A$ be a free arrangement in $V$ and $H \in \A$. Then
\[\pd_{\overline S} (J_\A^H)=\pd_S(\Omega^1(\A\setminus H)).\]
\end{lemma}

\begin{proof}
The statement follows from the exact sequence
\[0 \longrightarrow \Omega^1(\A\setminus H)
\lhook\joinrel\longrightarrow
\Omega^1(\A)
\longrightarrow
J_\A^H \longrightarrow 0
\]
together with the fact that $\pd_S(M)=\pd_{\overline S}(M)+1$ for any finitely generated $\overline S$-module $M$.
\end{proof}

\begin{proof}[Proof of Theorem \ref{thm:3.4}]
(1) This is known for experts, but we give a proof for completeness by using results above. The first equivalence is an immediate consequence of Lemma \ref{lem:3.5}.
We assume $\pd_{\overline S}(J_\A^H)=0$, equivalently $\A \setminus H$ is free, and prove $J_\A^H=\overline S$. 
In this case, since we already know that $\A\setminus H$ is free,
by Lemma \ref{P4} we have the exact sequence
\[
0 \longrightarrow D^{\ell-1}(\A \setminus H) \longrightarrow D^{\ell-1}(\A)
\stackrel{\rho_{\ell-1}^H} \longrightarrow D^{\ell-1}(\A^H) \longrightarrow 0.
\]
Let $y_1,\dots,y_{\ell-1}$ be a basis of $H^*$.
By Lemma \ref{P1} we have
\begin{align}
\label{eq.3.9.1}
D^{\ell-1}(\A^H)=Q(\A^H)\mathrm{Der}^{\ell-1}(\overline S)= 
(Q(\A^H) \cdot \overline S) \cdot (\partial_{y_1} \wedge \cdots \wedge \partial_{y_{\ell-1}}).
\end{align}
On the other hand,
when $\A$ is free, we see in the proof of Lemma \ref{lem:detformula} that
\begin{align}
\label{eq.3.9.2}
\frac 1 {Q(\A^H)} \mathrm{Im}(\rho_{\ell-1}^H)
= J_\A^H \cdot (\partial_{y_1} \wedge \cdots \wedge \partial_{y_{\ell-1}}).
\end{align}
Since $\rho_{\ell-1}^H$ is surjective,
the equations \eqref{eq.3.9.1} and \eqref{eq.3.9.2} 
prove $J_\A^H=\overline S$.

(2)
Consider the dual Euler sequence
\begin{align*}
0 \longrightarrow \Omega^1(\A) \longrightarrow \Omega^1(\A\setminus H) \longrightarrow \Omega^1(\A^H) \longrightarrow 0,
\end{align*}
where the right exactness follows from Lemma \ref{P4}.
Since $\Omega^1(\A)$ is a free $S$-module, by the above exact sequence we have
\begin{align}
    \label{3-7}
    \pd_S\big(\Omega^1(\A\setminus H)\big) \leq 1
    \Leftrightarrow \pd_S\big(\Omega^1(\A^H)\big)\leq 1
    \Leftrightarrow \pd_{\overline S}\big(\Omega^1(\A^H)\big)= 0.
\end{align}
Now by Lemma \ref{lem:3.5} it follows that
$\Omega^1(\A^H)$ is a free $\overline S$-module if and only if $\pd_{\overline S} (J_\A^H) \leq 1$ as desired.
\end{proof}

\begin{example}
Consider the arrangement $\A$ with $Q(\A)=x(y^2-x^2)(z^2-x^2)$ in Example \ref{ex:3-2}
and let $H=H_x$ and $L=H_{y-x}$.
We see that $J_\A^H=(y,z)$ and this ideal clearly has projective dimension $1$.
This reflects the fact that $\A \setminus H$ is not free but $\A^H$ is free.
On the other hand, we see that $J_\A^L=\R[x,y,z]/(y-x)$. This reflects the fact that $\A \setminus L$ is free.
\end{example}

\begin{example}
\label{exdim4}
Let $\A$ be the arrangement in $\R^4$ with 
\[
Q(\A)=x(y^2-x^2)(z^2-x^2)(w^2-x^2)(y+z)(y+w)(z+w).
\]
Let $H=H_x$ and $L=H_{y-x}$.
Below is the description of the ideals $J_\A^H$ and $J_\A^L$ (we will explain how we computed these ideals in Example \ref{ex4.7}):
\[
J_\A^H=(y,z,w)
\ \ 
\mbox{ and }
\ \ 
J_\A^L=(x+z,x+w).\]
It is clear that the ideal $J_\A^H$ has projective dimension $2$ and the ideal $J_\A^L$ has projective dimension $1$.
These reflect the fact that $\A^H$ is not free but $\A^L$ is free.
\end{example}

\begin{example}[Edelman--Reiner arrangement]
\label{ERexample}
Let $\A$ be the arrangement in $\R^5$ with
\[
\textstyle
Q(\A)=\big (\prod_{i=1}^5 x_i \big)\big (\prod(x_1 \pm x_2 \pm x_3\pm x_4 \pm x_5)\big),\]
where $\pm$ takes all possible combinations and $\A$ consists of 21 hyperplanes.
Let $H=H_{x_1+\cdots +x_5}$ and $L=H_{x_1}$.
It was shown by Edelman and Reiner \cite{ER} that $\A$ is free but $\A^H$ is not free.
Here are ideals $J_\A^H$ and $J_\A^L$ computed by Macaulay2 \cite{M2}:
\[
J_\A^H=(x_2,x_3,x_4,x_5),
\]
{\small
\[
J_\A^L=\big( (x_2^2\!-\!x_3^2)(x_2^2\!+\!x_3^2\!-\!x_4^2\!-\!x_5^2),
(x_2^2\!-\!x_4^2)(x_2^2\!-\!x_3^2\!+\!x_4^2\!-\!x_5^2),
(x_2^2\!-\!x_5^2)(x_2^2\!-\!x_3^2\!-\!x_4^2\!+\!x_5^2) \big).
\]
}

\noindent
The ideal $J_\A^H$ clearly has projective dimension $3$.
On the other hand,
one can check that
the projective dimension of $J_\A^L$ is equal to $1$, so $\A^L$ is free.
\end{example}

\section{Radicals of residue ideals}

In general, residue ideals could have complicated generators.
However, 
these ideals often admit geometrically meaningful primary decompositions.
Let us explain this by using the Edelman--Reiner arrangement.
Let $\A$ and $L=H_{x_1}$ be as in  Example \ref{ERexample}.
The ideal $J_\A^L$
is generated by the three polynomials
{\small
\[
(x_2^2-x_3^2)(x_2^2+x_3^2-x_4^2-x_5^2),
(x_2^2-x_4^2)(x_2^2-x_3^2+x_4^2-x_5^2),
(x_2^2-x_5^2)(x_2^2-x_3^2-x_4^2+x_5^2).
\]
}

\noindent
This ideal actually has the following primary decomposition
{\Small
\begin{align*}
& \big((x_2^2-x_3^2)(x_2^2+x_3^2-x_4^2-x_5^2),
(x_2^2-x_4^2)(x_2^2-x_3^2+x_4^2-x_5^2),
(x_2^2-x_5^2)(x_2^2-x_3^2-x_4^2+x_5^2)\big)\\
&=(x_2-x_3,x_4-x_5)\cap(x_2-x_3,x_4+x_5)\cap(x_2+x_3,x_4-x_5)\cap(x_2+x_3,x_4+x_5)\\
    &\hspace{14pt} \cap (x_2-x_4,x_3-x_5)\cap(x_2-x_4,x_3+x_5)\cap(x_2+x_4,x_3-x_5)\cap(x_2+x_4,x_3+x_5)\\
    &\hspace{14pt} \cap (x_2-x_5,x_3-x_4)\cap(x_2-x_5,x_3+x_4)\cap(x_2+x_5,x_3-x_4)\cap(x_2+x_5,x_3+x_4),
\end{align*}
}

\noindent
which shows that this ideal is a radical ideal and its zero set is the union of 12 codimension 2 subspaces in $L$.
In this and next sections,
we study radicals and primary decompositions of residue ideals and show that this phenomenon always happens in a certain sense.

\subsection{Radicals of residue ideals}
Recall that, for an arrangement $\A$ in $V$ and $H \in \A$,
we defined
\[\Xi(\A,H)=\{ X \in L(\A^H) \mid J_{\A_X}^H \ne \overline S\}.\]
We also write $\Xi_k(\A,H)=\{X \in \Xi(\A,H) \mid \dim X=\ell-k\}$
and write $\Xi^{\max}(\A,H)$ for the set of maximal elements in $\Xi(\A,H)$ with respect to inclusion.
We note that $\Xi(\A,H)$ is an ideal of $L(\A^H)$, that is, it satisfies,
for $X,Y \in L(\A^H)$, if $X \in \Xi(\A,H)$ and $Y \subset X$ then one has $Y \in \Xi(\A,H)$ by Lemma \ref{2.8}.
Our main result in this section is the following which shows that the zero set of $J_\A^H$ is the union of the subspaces in $\Xi(\A,H)$.

\begin{theorem}
    \label{thm:4.1}
    Let $\A$ be an arrangement in $V$ and $H \in \A$.
    Then
    \[\sqrt{J_\A^H}=\bigcap_{X \in \Xi(\A,H)} P_X=\bigcap_{X \in \Xi^{\max}(\A,H)} P_X.\]
\end{theorem}

\begin{proof}
The second equality is obvious. So we prove the first equation.
Recall that $P_X \subset \overline S$ denotes the defining ideal of the subspace $X$ of $H$.
Since the primeness of $P_X$ is preserved under field extensions,
to prove the statement,
we may assume that $\K$ is algebraically closed.
For a point $\bm p \in H$, we write $M_{\bm p}$ for the localization of an $\overline S$-module $M$ at the maximal ideal $\mideal_{\bm p} \subset \overline S$ corresponding to $\bm p$ and write $\A_{\bm p}=\A_{\mideal_{\bm p}}=\{ L \in \A \mid \bm p \in L\}$.
Since the zero set of $J_\A^H$ is given by
$
\{ \bm p \in H \mid (J_\A^H)_{\bm p} \ne \overline S_{\bm p}\},
$
to prove the statement,
it suffices to prove 
\begin{align}
\label{eq:4-1}
     \left( J_\A^H\right )_{\bm p} \ne \overline S_{\bm p} \ \ \ \mbox{ if and only if }\ \ \ \bm p \in \bigcup_{Y \in \Xi(\A,H)} Y
\end{align}
for any $\bm p \in H$.
We note that $J_{\{H\}}^H=\overline S$ so $H$ cannot be contained in $\Xi(\A,H)$.

Let $\bm p \in H$ and let $X$ be the smallest subspace of $L(\A)$ containing $\bm p$.
Then, since $\Xi(\A,H)$ is an ideal of $L(\A^H)$,
one has $\bm p \in \bigcup_{Y \in \Xi(\A,H)} Y$ if and only if $X \in \Xi(\A,X)$.
Also by the choice of $X$ we have $\A_{\bm p}=\A_X$ and 
\[ (J_\A^H)_{\bm p}=(J_{\A_X}^H)_{\bm p}\]
by Lemma \ref{2.12}.
Also, since $J_{\A_X}^H$ admits a generating set consisting of polynomials $f$ vanishing at $\bm p$ unless $J_{\A_X}^H=\overline S$ by Corollary \ref{2.6},
$J_{\A_X}^H$ does not contain a unit of $\overline S_{\bm p}$ unless $J_{\A_X}^H=\overline S$.
Hence
\[
(J_\A^H)_{\bm p} \ne \overline S_{\bm p}
\Leftrightarrow
(J_{\A_X}^H)_{\bm p} \ne \overline S_{\bm p}
\Leftrightarrow
J_{\A_X}^H \ne \overline S
\Leftrightarrow
X \in \Xi(\A,H).
\]
By the choice of $X$,
this proves the desired equivalence \eqref{eq:4-1}.
\end{proof}

In the rest of this section,
to have a better understanding of the description in Theorem \ref{thm:4.1}, we study properties of $\Xi(\A,H).$
We first give another description of $\Xi(\A,H)$ when $\A$ is a free arrangement.
The next lemma is a reformulation of Theorem \ref{thm:3.4}(1).

\begin{lemma}
    \label{lem:4.2}
    Let $\A$ be an arrangement, $H \in \A$ and $X \in L(\A^H)$.
    If $\A_X$ is free, then $X \in \Xi(\A,H)$ if and only if $\A_X\setminus H$ is not free.
\end{lemma}

For a free arrangement $\A$ and $H \in \A$, define
\[
\NFT(\A,H)=\{X \in L(\A^H) \mid \A_X \setminus H \mbox{ is not free}\}.
\]
We also define $\NFT_k(\A,H)$ and $\NFT^{\max}(\A,H)$ in the same way as for $\Xi(\A,H)$.
Since the freeness of $\A$ and $\A\setminus H$ implies the freeness of $\A^H$ by Terao's restriction theorem or Theorem \ref{thm:3.4},
the set $\NFT(\A)$ is the set of all subspaces $X \in L(\A^H)$ such that $(\A_X\setminus H,\A_X,\A_X^H)$ is \underline not a \underline free \underline triple.
Since a localization of a free arrangement is again free by 
Lemma \ref{localizationfree},
the set $\Xi(\A,H)$ is nothing but the set $\NFT(\A,H)$ by Lemma \ref{lem:4.2}.

\begin{cor}
    \label{4.3}
    If $\A$ is a free arrangement, then for any $H \in \A$, one has
    \[ \Xi(\A,H)=\NFT(\A,H).\]
\end{cor}

\begin{rem}
    \label{remNFT}
It was proved in \cite{A4} that,
when $\A$ is free, the deletion $\A \setminus H$ is free if and only if $\chi(\A_X^H;t)$ divides $\chi(\A_X;t)$ for all $X \in L(\A^H)$,
where $\chi(\A;t)$ is the characteristic polynomial of $\A$.
Thus $\NFT(\A,H)$ is combinatorially determined when $\A$ is free.
\end{rem}

\begin{lemma}
\label{lem:4.5}
Let $\A$ be an arrangement in $V$ and $H \in \A$.
The set $\Xi(\A,H)$ does not contain any subspace of dimension $\geq \ell-2$, that is, $\Xi(\A,H)=\bigcup_{k=3}^\ell \Xi_k(\A,H).$
\end{lemma}

\begin{proof}
If $X \in L(\A^H)$ has dimension $\geq \ell-2$,
then both $\A_X$ and $\A_X \setminus H$ are free because any arrangement in $\K^2$ is free.
Then the statement follows from Lemma \ref{lem:4.2}. 
\end{proof}

The next fact will be proved later in Corollary \ref{cor:assprimecore} in a more general form,
but we introduce this here since it is convenient to compute examples.

\begin{lemma}
    \label{lem:4.6}
    Let $\A$ be an arrangement and $H \in \A.$
    Then $\Xi^{\max}(\A,H) \subset L(\A^H_{\mathrm{core}})$.
    In particular, 
    \[ \sqrt {J_\A^H}= \bigcap_{X \in L(\A^H_{\mathrm{core}}),\ J_{\A_X}^H \ne \overline S} P_X.
    \]\end{lemma}

\begin{example}
    \label{ex4.7}
    Let $\A$ be the arrangement in Example \ref{exdim4} with
\[
Q(\A)=x(y^2-x^2)(z^2-x^2)(w^2-x^2)(y+z)(y+w)(z+w).
\]
Let $H=H_x$ and $L=H_{y-x}$. We note that $\A$ is a free arrangement with exponents $(1,3,3,3).$ We compute $J_\A^H$ and $J_\A^L$ using results in this and previous sections.

For $H=H_x$, we have
\[
\A_{\mathrm{core}}^H=\{H \cap H_{y-x}=H\cap H_y,H \cap H_{z-x}=H\cap H_z,H \cap H_{w-x}=H\cap H_w\}
\]
and,
since $\Xi(\A,H)$ can contain only subspaces of dimension $\leq 4-3$, by Lemma \ref{lem:4.6} candidates for elements in $\Xi(\A,H)$ are
$$
H\cap H_y\cap H_z,
H\cap H_y\cap H_w,
H\cap H_z\cap H_w,
H\cap H_y\cap H_z\cap H_w=\{\bm 0\}.
$$
One can check that $\A \setminus H$ is not free but $\A_{H\cap H_y\cap H_{z}}\!\!\setminus\! H$ is free. Since $\A$ is symmetric with respect to $y,z,w$,
these guarantee 
$\Xi(\A,H)=\NFT(\A,H)=\{\{\bm  0\}\}$
and hence $\sqrt{J_\A^H}=(y,z,w)$.
On the other hand,
since $\A$ has exponents $(1,3,3,3)$, by Lemma \ref{lem:3.3sononi} the ideal $J_\A^H$ must be generated by three polynomials of degree $1+3+3-|\A^H|=1$. Since an ideal generated by linear forms must coincide with its radical, $J_\A^H$ must coincide with $(y,z,w)$.

For $L=H_{y-x}$, we have
\[
\A_{\mathrm{core}}^L=\{H \cap H_{x},H \cap H_{z+x},H \cap H_{w+x}\}.
\]
and in this case 
$\A_{H\cap H_{z+x}\cap H_{w+x}}\!\!\setminus\! H$ is not free
but 
$\A_{H\cap H_x\cap H_{z+x}}\!\!\setminus \! H$
and
$\A_{H\cap H_x\cap H_{w+x}}\!\!\setminus \! H$ are free.
These say $\Xi^{\mathrm{max}}(\A,L)=\{H \cap H_{z+x}\cap H_{w+x}\}$
and $\sqrt{J_\A^L}=(z+x,w+x)$.
Also, since $J_\A^L$ is generated by polynomials of degree $1$ by Lemma \ref{lem:3.3sononi}, we conclude that 
$J_\A^L=(z+x,w+x)$.
\end{example}

It would be natural to ask when $J_\A^H$ is radical.
We will discuss this problem more in \S 5.4.

\begin{example}
    \label{nonradicalex}
A residue ideal is not necessarily a radical ideal.
If $\A$ is the arrangement in $\R^3$ with $Q(\A)=xyz(y^2-x^2)(z^2-x^2)(z-y)$
and $H=H_x$,
then $D(\A)$ is free with basis $\theta_E,y(y^2-x^2) \partial_y +z(z^2-x^2) \partial_z, (z^2-x^2)(z-y)z \partial_z$ by Theorem \ref{Saitocriterion}.
Thus 
\[
J_\A^H=
\frac 1 {yz(z-y)} I_2 \begin{pmatrix}
    y & y^3 & 0  \\ z & z^3 &(z-y)z^3
\end{pmatrix}
=(y+z,z^2) \ne \sqrt {(y+z,z^2)}=(y,z).\]
\end{example}

\subsection{Some commutative algebra}

To discuss more about algebraic properties of residue ideals,
here we recall basic results on 
commutative algebra relating primary decompositions and the Cohen--Macaulay property.
Although primary decompositions themselves will be discussed mainly in the next section, we need some related properties to explain easy consequences of our formula of $\sqrt{J_\A^H}$ in this section.
We refer the reader to \cite{AM} and \cite{Ei} for basics on primary decompositions as well as undefined terminology on commutative algebra such as depth, height, Krull dimension, etc.

For an ideal $I$ in $S$,
a prime ideal $P \subset S$ is said to be an {\bf associated prime} of $S/I$ (or $I$) if there is an element $g \in S$ such that $P=\{f \in S \mid fg \in I\}$.
The set of all associated primes of $S/I$ is known to be a finite set and will be denoted by $\Ass(S/I)$ (or $\Ass(I)$).
An ideal $Q \subset S$ is said to be {\bf $P$-primary} if $\Ass(S/Q)=\{P\}$, where $P$ is a prime ideal of $S$.
Note that if $Q$ is $P$-primary then $\sqrt Q=P$.
The set $\Ass(S/I)$ contains the set $\Min(S/I)$ of minimal prime ideals of $S$ that contain $I$ (with respect to inclusion).

Lasker--Noether theorem guarantees that
any ideal in $S$ can be written as the intersection of finitely many primary ideals.
More precisely,
let $\Ass(S/I)=\{P_1,\dots,P_t\}$, where $P_1,\dots,P_t$ are distinct.
Then there exist $P_i$-primary ideals $Q_i$ ($i=1,2,\dots,t$) such that 
\begin{align}
\label{primarydec}    
I= Q_1 \cap Q_2 \cap \cdots \cap Q_t.
\end{align}
Moreover if $P_i \in \Min(S/I)$ then the corresponding $P_i$-primary component is uniquely determined and is given by $Q_i=I_{P_i}\cap S$.
We call the presentation \eqref{primarydec} a {\bf canonical primary decomposition} of $I$.
By \eqref{primarydec},
we have a prime decomposition $\sqrt I = \bigcap_{P \in \Ass(S/I)} P$,
but this presentation may not be irredundant, that is, some $P_i$ may contain other $P_j$,
and the unique irredundant prime decomposition of $\sqrt I$ is given by
$\sqrt I=\bigcap_{P \in \Min(S/I)} P$.
A prime ideal in $\Ass(S/I)$ that is not contained in $\Min(S/I)$ is called an {\bf embedded prime} of $S/I$.

Next, we recall a few basic facts on associated primes and the Cohen--Macaulay property.
We denote by $\dim (S/I)$ the Krull dimension of a $\K$-algebra $S/I$.
The following fact is standard in commutative algebra.
See e.g.\ \cite[Proposition 1.2.13]{BH}.

\begin{lemma}
\label{lem:4.9}
Let $I$ be a homogeneous ideal in $S$.
If $P \in \Ass(S/I)$, then $\depth(S/I) \leq \dim (S/P)$.
\end{lemma}

A graded $\K$-algebra $S/I$ (or a homogeneous ideal $I$) is said to be {\bf Cohen--Macaulay} if $\depth(S/I)=\dim(S/I)$.
The following fact is also well-known and easily follows from Lemma \ref{lem:4.9}.

\begin{lemma}
\label{lem:4.10}
If a homogeneous ideal $I \subset S$ is Cohen--Macaulay, then $I$ is unmixed,
that is, $\dim S/P=\dim S/I$ for all $P \in \Ass(S/I)$.
In particular, $S/I $ has no embedded primes.
\end{lemma}

\subsection{Algebraic properties of $J_\A^H$ and quick applications}

Our original motivation of considering the residue ideal $J_\A^H$ is to understand when $\A^H$ becomes free assuming that $\A$ is free.
The next proposition is a reformulation of Theorem \ref{thm:3.4} but clarifies what kind of algebraic property of $J_\A^H$ should be considered to study such a problem.

\begin{prop}
\label{prop:4.12}
Let $\A$ be a free arrangement in $V$ and $H \in \A$. 
The arrangement $\A^H$ is free if and only if either $J_\A^H=\overline S$ or
    $\overline S/J_\A^H$ is Cohen--Macaulay of Krull dimension $\ell-3$, equivalently, $J_\A^H$ is a height two Cohen--Macaulay ideal.
\end{prop}

\begin{proof}
Since $J_\A^H=\overline S$ if and only if $\A \setminus H$ is free by Theorem \ref{thm:3.4},
we assume $J_\A^H \ne \overline S$ and prove that $\A^H$ is free if and only if 
$\overline S/J_\A^H$ is Cohen--Macaulay of Krull dimension $\ell-3$.
Since $\A\setminus H$ is not free, we have $\Xi(\A,H)=\NFT(\A,H) \ne \varnothing$.
Then by Lemma \ref{lem:4.5} we have
\[
\dim (\overline S /J_\A^H) = \max \{ \dim X \mid X \in \Xi(\A,H)\} \leq \ell -3,
\]
and hence 
\[
\pd_{\overline S}(J_\A^H)
= (\ell-2)-\depth (\overline S/J_\A^H) \geq 
(\ell-2)-\dim (\overline S/J_\A^H) \geq 1,
\]
where we use the Auslander--Buchsbaum formula for the first equality and use the fact that depth is smaller than or equal to Krull dimension in the second inequality.
Since $\A^H$ is free if and only if $\pd_{\overline S}J_\A^H \leq 1$ by Theorem \ref{thm:3.4},
it follows from the above inequality that 
$\A^H$ is free if and only if $\depth(\overline S/J_\A^H)=\dim(\overline S/J_\A^H)=\ell-3$, proving the desired statement.
\end{proof}

The proposition gives an interesting connection between free arrangement theory and subspace arrangements.
Since Cohen--Macaulay ideals are unmixed,
in the situation of Proposition \ref{prop:4.12} the radical of $J_\A^H$ becomes the ideal of codimension two subspace arrangements of $H$ unless $J_\A^H=\overline S$.

Finally,
we give two quick applications of the results explained in this section.
Let $\A$ be a free arrangement and $H \in \A$.
Suppose $\A\setminus H$ is not free.
Proposition \ref{prop:4.12} says that
if $J_\A^H$ is not a height two unmixed ideal, then $\A^H$ is not free
because Cohen--Macaulay ideals must be unmixed.
This leads to the following statement
which can be used to disprove the freeness of $\A^H$ when $\A$ is free.

\begin{prop}
\label{nonfree}
Let $\A$ be a free arrangement in $V$ and $H \in \A$. If $\NFT^{\max}(\A,H)$ contains a subspace of dimension $\leq \ell-4$, then $\A^H$ is not free.
\end{prop}

\begin{cor}
Let $\A$ be a free arrangement in $V$ and $H \in \A$. If $\A^H$ is free, then either 
\begin{itemize}
\item[(1)]
$\A_X \setminus H$ is free 
for all $X \in L(\A^H)$, or 
\item[(2)]
if $\A_X \setminus H$ is not free for some $X \in L(\A^H)$, then there is an $(\ell-3)$-dimensional subspace $Y \in L(\A^H)$ such that $X \subset Y$ and that $\A_Y \setminus H$ is not free.
\end{itemize}
\label{freenonfree}
\end{cor}

These results show the importance of local freeness for the freeness of $\A^H$. In \cite{A4}, it was proved that if $\A$ is free and $\A\setminus \{H\}$ is locally free along $H$, then 
$\A^H$ and $\A \setminus H$ are free if $\chi(\A';t)$ divides $\chi(\A;t)$, where $\chi(\A;t)$ is the characteristic polynomial of $\A$.
The above corollary is 
a generalization of this result.

\begin{example}
Consider
arrangements $\A$ and $H \in \A$ given in Examples \ref{exdim4} and \ref{ERexample}.
We see that $\A$ is free but its restriction $\A^H$ is not free.
In both cases, the non-freeness can be seen from $\NFT(\A,H)=\Xi(\A,H)$.
Indeed, in both cases, the only element in $\NFT(\A,H)$ is the origin $\bm 0$ and it is clear that the condition of Proposition \ref{nonfree} is satisfied.
\end{example}

Let us show that $\NFT(\A,H)$ can be also used to prove the freeness of $\A^H$ in a certain situation.

\begin{prop}
Let $\A$ be a free arrangement in $V$ with exponents $(1,a,\ldots,a)$ and let $H \in \A$ with $|\A^H|=(\ell-2)a$. If $\NFT_{3}(\A,H) 
\neq \varnothing$, then $\A^H$ is free.
\label{restfree}
\end{prop}

\begin{proof}
It follows from Lemma \ref{lem:3.3sononi} and the assumption of the theorem that $J_\A^H$ is generated by linear forms.
Thus $J_\A^H=P_X$ for some subspace $X \subset H$.
Since $\NFT_3(\A,H) \ne \varnothing$,
it follows from Theorem \ref{thm:4.1} that $X$ has dimension $\ell-3$.
Thus $J_\A^H=P_X$ has projective dimension $1$ over $\overline S$.
\end{proof}

We note that the above proof 
is nothing but the method which we used to compute $J_\A^L$ in Example \ref{ex4.7}.

\section{Primary decompositions and associated primes}

In this section,
we study primary decompositions of residue ideals.

\subsection{Primary decompositions of residue ideals}

Our first goal is to complete the proof of Theorem \ref{thm1.3} in the introduction.
We begin with the following consequence of Theorem \ref{thm:4.1}.

\begin{lemma}
    \label{5.1}
    Let $\A$ be an arrangement in $V$ and $H \in \A$.
    \begin{itemize}
        \item [(1)] $\Min(\overline S/J_\A^H)=\{ P_X \mid X \in \Xi^{\max}(\A,H)\}$.
        \item[(2)] $\Ass(\overline S/J_\A^H) \subset \{ P_X \mid X \in \Xi(\A,H)\}$.
        \item[(3)] If $X \in \Xi^{\max}(\A,H)$ and $\A_X$ is free,  then $J_{\A_X}^H$ is a complete intersection.
    \end{itemize}
\end{lemma}

\begin{proof}
Statement (1) is an immediate consequence of Theorem \ref{thm:4.1}.
We prove (2).
Let $P \in \Ass(\overline S / J_\A^H)$
and let $X \in L(\A^H)$ be the subspace such that $\A_X=\A_P$.
Then $(J_{\A_X}^H)_P=(J_\A^H)_P$ by Lemma \ref{2.12} and $X \in \Xi(\A_X,H)$ since $(J_\A^H)_P \ne \overline S_P$.
Also, since every hyperplane $L \in \A_X$ belongs to $\tilde P$, where $\tilde P$ is the prime ideal of $S$ with $S/\tilde P=\overline S/P$,
and $X=\bigcap_{L \in \A_X}L$, we have $P_X \subset P$.
We prove $P\subset P_X$.

By Corollary \ref{2.6}
we have $J_{\A_X}^H=J_{(\A_X)^e}^{H/X}\overline S$.
Then since $J_{(\A_X)^e}^{H/X}$ is an ideal of $\mathrm{sym}((H/X)^*)$ and since a primary decomposition of $J_{(\A_X)^e}^{H/X}$ coincides with that of $J_{\A_X}^H$, any associated prime of $J_{\A_X}^H$ is contained in $P_X$.
On the other hand, 
since 
$(J_\A^H)_P=(J_{\A_X}^H)_P$
and since, for any ideal $J \in \overline S$,
there is a bijection between
the associated primes of $(\overline S/J)_P$
and the associated primes of $\overline S/J$
contained in $P$,
the ideal $P$ must be an associated prime of $J_{\A_X}^H$. 
Since any associated prime of $J_{\A_X}^H$ is contained in $P_X$,
this proves $P \subset P_X$ as desired.

(3)
Recall that an ideal $J$ of $\overline S$ is a complete intersection if it is generated by $k$ elements and $\overline S/J$ has Krull dimension $\ell-1-k$ for some $k$.
Note that, by Krull's height theorem 
an ideal $J \subset \overline S$ generated by at most $k$ elements satisfying $\dim (\overline S/J)=\ell-1-k$ is automatically a complete intersection.
Assume $X \in \Xi^{\max}(\A,H)$.
By Lemma \ref{lem:3.3sononi} and Corollary \ref{2.6}, the ideal $J_{\A_X}^H$ is generated by at most $\ell-1-\dim X$ polynomials.
On the other hand,
since
\[
\Xi^{\max}(\A_X,H)=\{L \in \Xi^{\max}(\A,H)\mid X \subset L\}=\{X\},\]
we have $\sqrt {J_{\A_X}^H}=P_X$ and  $\dim(\overline S/J_{\A_X}^H)=\dim X$.
These prove that $J_{\A_X}^H$ is a complete intersection.
\end{proof}

\begin{prop}
\label{5.2}
Let $\A$ be an arrangement in $V$, $H \in \A$ and $X \in L(\A^H)$.
Let $J_\A^H=\bigcap_{P_Y \in \mathrm{Ass}(\overline S/J_\A^H)} Q_Y$ be a canonical primary decomposition of $J_\A^H,$ where $Q_Y$ is a $P_Y$-primary ideal with $Y \in \Xi(\A,H)$.
Then 
\[ J_{\A_X}^H = \bigcap_{P_Y \in \mathrm{Ass}(\overline S/J_\A^H),\ \! Y \supset X}Q_Y.\]
In particular, if $X \in \Xi^{\max}(\A,H)$ then $J_{\A_X}^H$ is $P_X$-primary.
\end{prop}

\begin{proof}
We first note that for prime ideals $P' \supset P$ and a $P$-primary ideal $Q$, one has $Q=Q_{P'}\cap S$ since $Q_{P} \supset Q_{P'} \supset Q$ and since $Q_P \cap S=Q$ as $P$ is a minimal prime of $Q$.
Let $\mathrm{Ass}_X=\{P_Y \in \mathrm{Ass}(\overline S/J_\A^H) \mid Y \supset X\}.$
Then $Q_Y=S \cap (Q_Y)_{P_X}$ for any $P_Y \in \mathrm{Ass}_X$.
Hence by Lemma \ref{2.12} we have
\[ \bigcap_{P_Y \in \mathrm{Ass}_X} Q_Y=S \cap \left( \bigcap_{P_Y \in \mathrm{Ass}_X} Q_Y\right)_{P_X}=S \cap (J_\A^H)_{P_X} =S \cap (J_{\A_X}^H)_{P_X}.\]
Thus, to prove the desired equality, it remains to prove $(J_{\A_X}^H)_{P_X}\cap S=J_{\A_X}^H.$

Let $J_{\A_X}^H=Q_1 \cap \cdots \cap Q_m$ be a canonical primary decomposition of $J_{\A_X}^H$.
Since $J_{\A_X}^H=J_{\A^{V/X}_X}^{H/X} \overline S$
and $J_{\A^{V/X}_X}^{H/X}$ is an ideal of $\mathrm{sym}((H/X)^*)$,
each $Q_i$ must be contained in $P_X$.
Hence $(Q_i)_{P_X} \cap S=Q_i$ for all $i$ and we have
\[
(J_{\A_X}^H)_{P_X} \cap S=
\big( (Q_1)_{P_X}\cap S\big) \cap \cdots \cap \big( (Q_m)_{P_X}\cap S\big)=J_{\A_X}^H,
\]
as desired.
\end{proof}

The next corollary is an immediate consequence of Proposition \ref{5.2}.

\begin{cor}
    \label{cor:newprimarydecomp}
    Let $\A$ be an arrangement in $V$ and $H \in \A$.
    Then
\[ J_\A^H= \bigcap_{P_X \in \mathrm{Ass}(\overline S/J_\A^H)} J_{\A_X}^H.\]
In particular, if $\overline S/J_{\A}^H$ has no embedded primes, then the above presentation is a primary decomposition of $J_\A^H$.
\end{cor}

We have already seen that if both $\A$ and $\A^H$ are free, then $J_\A^H$ is either trivial or a height two unmixed ideal. Thus, in this special case, Lemma \ref{5.1}(3) and Proposition \ref{5.2} give the following statement.

\begin{cor}
    \label{cor:newprimarydecompfree}
    Let $\A$ be a free arrangement in $V$ and $H \in \A$. If $\A^H$ is free, then
\[ J_\A^H= \bigcap_{X \in \Xi_3(\A,H)} J_{\A_X}^H\]
and each $J_{\A_X}^H$ with $X \in \Xi_3(\A,H)$ is a height two complete intersection.
\end{cor}

We finally note that to consider associated primes of $J_\A^H$, it is enough to consider subspaces in $L(\A_{\mathrm{core}}^H).$

\begin{cor}
    \label{cor:assprimecore}
    Let $\A$ be an arrangement in $V$ and $H \in \A$. One has
    \[ \mathrm{Ass}(\overline S/J_\A^H) \subset \{ P_X \mid X \in L(\A_{\mathrm{core}}^H)\}.\]
\end{cor}

\begin{proof}
Let $P_X \in \mathrm{Ass}(\overline S/J_\A^H)$ with $X \in L(\A^H)$.
Since
$ (\A_X)^H_{\mathrm{core}}=(\A_{\mathrm{core}}^H)_X$,
if we write $Y=\bigcap_{L \in (\A_X)^H_{\mathrm{core}}} L$, then it follows from Proposition \ref{prop:2.10} that
\[J_{\A_X}^H=J_{(\A_X)_Y}^H=J_{\A_Y}^H.\]
But since $P_X \in \mathrm{Ass}(\overline S/J_\A^H)$ the above equality and Proposition \ref{5.2} prove $X=Y$ which implies $X \in L(\A_{\mathrm{core}}^H).$
\end{proof}
\begin{rem}
The results in this section show that primary decompositions of residue ideals are similar to those of Jacobian ideals.
For example, about Proposition \ref{5.2} and Corollary \ref{cor:newprimarydecomp},
essentially the same statements for Jacobian ideals appear in \cite[Proposition 5.3 and Lemma 5.4]{DST}.
Also, a hyperplane arrangement $\A$ in $V$ is free if and only if its Jacobian $J_\A$ is a height two Cohen--Macaulay ideal, and in that case $J_\A= \bigcap_{X \in L_{2}(\A)} J_{\A_X}$ is an intersection of complete intersections where $L_k(\A)$ is the set of $(\ell-k)$-dimensional subspaces in $L(\A)$. See \cite[Proposition 4.3]{MN}.
\end{rem}

\subsection{Radicality of residue ideals}
Next, we study when $J_\A^H$ is radical.
We begin with the three dimensional case.
Recall that if $\A$ is a free arrangement in $\K^3$, 
then the only possible element of $\Xi(\A,H)$ is $\{\bm 0\}$.
Thus $J_\A^H$ is either $\overline S$ or $\mideal_{\overline S}$-primary by Lemma \ref{5.1},
where $\mideal_{\overline S} $ is the graded maximal ideal of $\overline S=\K[x_1,x_2,x_3]/(\alpha_H)$.

\begin{lemma}
    \label{5.10}
Let $\A \ne \varnothing$ be a free arrangement in $\K^3$ with exponents $(1,a,b)$ and $H \in \A.$
Assume that $\A \setminus H$ is not free. Then
\begin{itemize}
    \item [(1)] we have $\#\A^H \leq \min \{a,b\}$.
    In particular, $\# \A^H \leq \frac 1 2 (\#\A-1)$.
    \item [(2)] $J_\A^H$ is a prime ideal $\Leftrightarrow$ $\#\A^H=a=b$ $\Leftrightarrow$ $\#\A^H=\frac 1 2 (\#\A-1)$.
\end{itemize}
\end{lemma}

\begin{proof}
Statement (1) follows from Corollary 1.2 (1) in \cite{A} and Terao's deletion theorem. Here we give another proof using 
$J_\A^H$.
    Since $\A\setminus H$ is not free, $J_\A^H$ has height at least two by Lemma \ref{lem:4.5} and hence $J_\A^H$ is not a principal ideal. Then by Lemma \ref{lem:3.3sononi} the ideal $J_\A^H$ is generated by two polynomials of degrees $a+1-\#\A^H$ and $b+1-\#\A^H$, and these degrees must be positive. This proves $\#\A^H \leq \min \{a,b\}$.
    Also, since $\#\A=1+a+b$, we have  \[2\times (\#\A^H) \leq a+b = (\#\A-1),\]
    completing the proof of (1).
    Finally, since $\sqrt {J_\A^H}=\mideal_{\overline S}$, $J_\A^H$ is a prime ideal if and only if $a+1-\#\A^H=b+1-\#\A^H=1$, which proves (2).
\end{proof}

Using the previous lemma,
we prove the following sufficient condition for the radicality of $J_\A^H$.

\begin{theorem}
    \label{5.12}
    Let $\A$ be a free arrangement in $V$ and $H \in \A$. Assume 
    that 
    \[
    (\#) \hspace{50pt}
    \mathrm{mult}^\A(L) \leq 2 \mbox{ for any } L \in \A^H.\hspace{80pt}
    \]
Then
\begin{itemize}
    \item[(1)] if $X \in \Xi_3(\A,H)$ and $\A_X$ is free, then $J_{\A_X}^H$ is a prime ideal.
    \item[(2)]
    If $\A$ and $\A^H$ are free, then $J_\A^H$ is a radical ideal.
\end{itemize}\end{theorem}

\begin{proof}
Let $X \in \Xi_3(\A,H)$ and assume that $\A_X$ is free.
Since the condition (\#) says $\#\A_X^H \geq \frac 1 2 (\# \A_X-1)$ for any $X \in L(\A^H)$,
the statement (1) follows from Lemma \ref{5.10}.
The statement (2) is an immediate consequence of (1) and Corollary \ref{cor:newprimarydecompfree}.
\end{proof}

Theorem \ref{5.12} explains why the ideal $J_\A^L$ is radical in Examples \ref{exdim4} and \ref{ERexample}.
Another class of examples that satisfies the condition $(\#)$ is given by graphic arrangements,
which we discuss in \S 7.

\subsection{Associated primes and height two unmixedness}

We see in \S 4.3 that,
if an arrangement $\A$ is free, then
one can show the non-freeness of $\A^H$ by checking that $J_\A^H$ is not a height two unmixed ideal.
In this subsection, 
we study associated primes of $J_\A^H$ and discuss when $J_\A^H$ is a height two unmixed ideal.

Let $\A$ be an arrangement in $V$.
As we see in Section 3,
$\Omega^1(\A)$ is a reflexive $S$-module. So $\pd_S(\Omega^1(\A))\leq \ell-2$.
More precisely, since $\Omega^1(\A)$ and $\Omega^1(\A^e)$ have the same projective dimension, we actually have
\[\textstyle
\pd_S(\Omega^1(\A))\leq \ell-2 -\dim X_\A,
\]
where $X_\A=\bigcap_{H \in \A} H$.
We say that $\Omega^1(\A)$ has the {\bf maximal projective dimension} if
$\pd_S(\Omega^1(\A))= \ell-2 -\dim X_\A$.
Thus, for an essential arrangement $\A$, we say that $\Omega^1(\A)$ has the maximal projective dimension if $\pd_S(\Omega^1(\A))=\ell-2$ and,
for a non-essential arrangement $\A$,
we say that $\Omega^1(\A)$ has the maximal projective dimension if $\Omega^1(\A^e)$ has the maximal projective dimension.

\begin{prop}
\label{5.5}
Let $\A$ be an arrangement in $V$, $H \in \A$ and 
 $X \in \Xi(\A,H)$.
\begin{itemize}
    \item[(1)] If $P_X \in \Ass(\overline S/J_\A^H)$, 
then $\Omega^1(\A_X\setminus H)$ has the maximal projective dimension.
\item[(2)] The converse of (1) holds if $\Omega^1(\A_X)$ does not have a maximal projective dimension.
\end{itemize}
In particular, if $\A$ is free then $P_X \in \Ass(\overline S/J_\A^H)$
if and only if $\Omega^1(\A_X \setminus H)$ has the maximal projective dimension.
\end{prop}

\begin{proof}
If $X_{\A_X \setminus H} \ne X$, then $(\A_X)_{\mathrm{core}}^H = \varnothing$, implying $J_{\A_X}^H=\overline S$, 
a contradiction. Thus 
$X_{\A_X \setminus H}=X$.
Since Lemma \ref{2.12} says that $P_X \in \Ass(\overline S/J_\A^H)$ if and only if $P_X \in \Ass(\overline S/J_{\A_X}^H)$,
by Corollary \ref{2.6} it suffices to prove the statement when $\A$ is essential and $X=\bm 0$.
Under this assumption $P_X$ is the graded maximal ideal of $\overline S$, and we have $P_X \in \Ass(\overline S/J_\A^H)$ if and only if $\depth (\overline S/J_\A^H)=0$.
The latter condition is equivalent to
$\pd_{\overline S}(J_\A^H)=\dim \overline S-1=\ell-2$.
Consider the exact sequence
\[
0 \longrightarrow \Omega^1(\A \setminus H) \lhook\joinrel\longrightarrow \Omega^1(\A) \longrightarrow J_\A^H \longrightarrow 0.
\]
Since $\pd_S(\Omega^1(\A)) \leq \ell-2$ and $\pd_S(J_\A^H)=\pd_{\overline S}(J_\A^H)+1$,
if $\pd_{\overline S} (J_{\A}^H)=\ell-2$, then we must have $\pd_S(\Omega^1(\A \setminus H))=\ell-2$ by the above exact sequence,
proving (1).
Conversely,
if $\pd_S(\Omega^1(\A\setminus H))=\ell-2$ and $\pd_S(\Omega^1(\A))<\ell-2$, then again by the above exact sequence we have $\pd_S(J_\A^H)=\ell-1$, equivalently, $\pd_{\overline S}(J_\A^H)=\ell-2$. This proves (2).
\end{proof}

Now Lemma \ref{5.1}(2) and Proposition \ref{5.5} prove the following criterion.

\begin{cor}
Let $\A$ be a free arrangement in $V$ and $H \in \A$ such that 
$\A \setminus \{H\}$ is not free. Then 
the following conditions are equivalent.
\begin{itemize}
    \item[(1)] $J_\A^H$ is a height two unmixed ideal.
    \item[(2)] $\Omega^1(\A_X\setminus H)$ does not have a maximal projective dimension for all $X \in \Xi(\A,H)$ with $\dim X < \ell-3$.
\end{itemize}
\end{cor}

Our computational experiments tell us that it is hard to find arrangements $\A$ such that $\Omega^1(\A)$ has the maximal projective dimension unless $\A^e$ is an arrangement in $\K^3$.
Indeed, we do not know an essential arrangement $\A$ in $\K^\ell$ such that $\Omega^1(\A)$ has the maximal projective dimension for $\ell \geq 6$.
This fact and Proposition \ref{5.5} suggest that $J_\A^H$ is unlikely to have an associated prime of large height.
This also suggests that the decomposition in Corollary \ref{cor:newprimarydecomp} might be useful since it enables us to compute residue ideals from local information of $\A$ if all the associated primes have low height.
It would be interesting to find theoretical evidence supporting our experiments, or more strongly, to find a useful criterion for arrangements $\A$ such that $\Omega^1(\A)$ has the maximal projective dimension.
At present, this seems to be a difficult question, so we pose the following question.

\begin{question}
Can we find an essential arrangement $\A$ in $\K^\ell$ such that $\Omega^1(\A)$ has the maximal projective dimension for any $\ell \geq 6$?
\end{question}

\begin{example}
\label{ex:maximalPD}
It was proved in \cite[Theorem 6.5]{MS} that if an arrangement $\A$ is locally free and $\pd_S(D(\A))=1$, then $\Omega^1(\A)$ has the maximal projective dimension.
On the other hand, by the SPOG theorem in \cite{A5}, if $\A$ is free but $\A \setminus H$ is not free, then $\pd_S(D(\A\setminus H))=1$.
By these results,
if $\A$ is free and if $\A \setminus H$ is not free but is locally free, then $\Omega^1(\A\setminus H)$ must have the maximal projective dimension.
The arrangements $\A$ and $\A\setminus H$ in Examples \ref{exdim4} and \ref{ERexample} satisfy these conditions, so in these examples $\Omega^1(\A\setminus H)$ has the maximal projective dimension when $\ell=4,5$.
\end{example}

\begin{example}
\label{exDiPasquale}
It could happen that $J_\A^H$ is unmixed of height two but is not Cohen-Macaulay. Let $\A$ be the hyperplane arrangement in $\R^5$ with
{\Small
\[
Q(\A)=x_0(x_1^2-x_0^2)(x_2^2-x_0^2)(x_3^2-x_0^2)(x_4^2-x_0^2)(x_1-x_2)(x_2-x_3)(x_3-x_4)(x_1+x_4)
\]
}

\noindent
and let $H=H_{x_0}$. This arrangement $\A$ was considered by  DiPasquale and it was proved in \cite[Proposition 6.12]{Di} that $\A$ is free but $\A^H$ is not free.
A computation by Macaulay2 tells
\[J_\A^H=(x_2x_3,x_2x_5,x_3x_4,x_4x_5)=(x_2,x_4) \cap (x_3,x_5),\]
which is a height two unmixed ideal but has projective dimension $2$.
\end{example}

\begin{rem}
\label{5.8}
    Let $\A$ be a free arrangement in $\K^4$ and $H \in \A$.
    In this case, if $J_\A^H$ is a height two unmixed ideal, then $J_\A^H$ is Cohen--Macaulay and hence $\A^H$ is free.
This is because, for any homogeneous ideal $J$ of $\overline S$, we have $\pd_{\overline S} J \leq 2$ since $\overline S=S/(\alpha_H)$ has Krull dimension $3$, and hence
\[\pd_{\overline S} J=2
\Leftrightarrow \depth (\overline S/J)=0 \Leftrightarrow \mideal_{\overline S} \in \Ass(\overline S/J).\]   \end{rem}

\begin{example}
Residue ideals could have an embedded prime.
Let $\A$ be an arrangement in $\mathbb C^4$ with 
\[Q(\A)=(x^4-w^4)(y^4-w^4)(z^4-w^4)(x-iy)(x+z)(y+z)w,\]
where $i$ is the imaginary unit, and set $H=H_w$.
It was proved in \cite[Proposition 5.2]{DW} that $\A$ is free but $\A^H$ is not free.
In this case, it is not hard to see that 
\[
\Xi(\A,H)=\NFT(\A,H) = \{L_{xy},L_{xz},L_{yz},\bm 0\}
\]
where $L_{xy},L_{xz},L_{yz}$ are the lines in $H$ defined by $x=y=0$, $x=z=0$ and $y=z=0$ respectively. In particular, $\Min(\overline S/J_\A^H)=\{(x,y),(x,z),(y,z)\}$.
On the other hand,
since $\A^H$ is not free,  $(x,y,z) \subset \overline S$ is an associated prime of $J_\A^H$ (see Remark \ref{5.8}).
\end{example}

We do not know if there is an arrangement $\A$ in $\R^\ell$ such that $J_\A^H$ has an embedded prime for some $H \in \A$.

\section{Quick Introduction to Stanley--Reisner theory}

The second theme of this paper is the module $\Omega^1(\A_G)$ of logarithmic $1$-forms of a graphic arrangement $\A_G$.
It turns out that the module $\Omega^1(\A_G)$
is closely related to the Stanley--Reisner theory.
In this section, we introduce results in Stanley--Reisner theory that are required to discuss properties of $\Omega^1(\A_G)$.
For more information on Stanley--Reisner theory,
see the books \cite{HH} and \cite{Stanley}.

\subsection{Simplicial complexes and squarefree monomial ideals}
A {\bf simplicial complex} on $[\ell]$ is a collection $\Delta$ of subsets of $[\ell]$ satisfying that $F \in \Delta$ and $G \subset F$ implies $G \in \Delta$.
We consider that the empty set $\varnothing$ and the set $\{\varnothing\}$ of empty set are different simplicial complexes.
For a simplicial complex $\Delta$ on $[\ell]$ and $F \subset [\ell]$,
the simplicial complex
\[\lk_\Delta(F)=\{G\setminus F \mid F \subset G \in \Delta\}\]
is called the {\bf link of $F$ in $\Delta$}
and the simplicial complex
\[
\Delta[F]=\{G \subset F \mid G \in \Delta\}
\]
is called the {\bf induced subcomplex} of $\Delta$ on $F$,
where we consider that $\lk_\Delta(F)=\varnothing$ if $F \not \in \Delta$.

For a subset $F \subset [\ell]$,
let $x^F= \prod_{i \in F} x_i$, where $x^\varnothing=1$.
For a simplicial complex $\Delta$ on $[\ell]$,
its {\bf Stanley--Reisner ideal} is the squarefree monomial ideal in $S=\K[x_1,\ldots,x_\ell]$ defined by 
\[I_\Delta=(x^F \mid F \subset [\ell],\ F \not \in \Delta)\]
and its {\bf Stanley--Reisner ring} is the ring $\K[\Delta]=S/I_\Delta$.
We note that, for any squarefree monomial ideal $I\subset S$,
the set $\Delta(I)=\{F \subset [\ell] \mid x^F \not \in I\}$ forms a simplicial complex and one has $I_{\Delta(I)}=I$.
So the correspondence $\Delta \to I_\Delta$ gives a bijection between the set of simplicial complexes on $[\ell]$ and the set of squarefree monomial ideals in $S$.

We actually need a relative version of Stanley--Reisner rings.
Let $\Delta \supset \Gamma$ be simplicial complexes on $[\ell]$.
The {\bf Stanley--Reisner module} of the pair $(\Delta,\Gamma)$ is the $S$-module
\[\K[\Delta,\Gamma]=I_\Gamma/I_\Delta.\]
In other words, Stanley--Reisner modules are modules of the form $I/J$ for some squarefree monomial ideals $I \supset J$.
Note that $\K[\Delta,\varnothing]=\K[\Delta]$ since $I_\varnothing =S$.

\subsection{Algebraic invariants of Stanley--Reisner rings and modules}
Algebraic properties of Stanley--Reisner rings and modules can be often determined from combinatorial or topological information of simplicial complexes.
Let us introduce some known results on this direction.

For a graded $S$-module $M=\bigoplus_{i \in \Z} M_i$, where $M_i$ is the degree $i$ graded component of $M$,
its {\bf Hilbert series} is the formal power series
\[
\Hilb(M,t)=\sum_{i \in \Z} (\dim_\K M_i) t^i.
\]
The following combinatorial formula for the Hilbert series of a Stanley--Reisner module is known.
See \cite[II, 1.4 Theorem]{Stanley}.

\begin{lemma}
    \label{lem:SRHilb}
Let $I \supset J$ be squarefree monomial ideals in $S$.
Then the set
\[
\bigcup_{x^F \in I\setminus J}
\big\{ x^F \cdot m \mid m \text{ is a monomial in }\K[x_i \mid i \in F] \big\}
\]
is a $\K$-basis of $I/J$.
In particular, we have
\[\Hilb(I/J,t)=\sum_{x^F \in I \setminus J} \frac {t^{\#F}} {(1-t)^{\#F}}.\]
\end{lemma}

For a finitely generated graded $S$-module $M$,
let $H_{\mideal_S}^i(M)$ be the $i$th local cohomology module of $M$ with respect to the maximal ideal $\mideal_S$ of $S$ (see \cite[\S 3]{BH} and \cite{Hu} for basics on local cohomology).
We write $\widetilde H_i(\Delta)$ (resp.\ $\widetilde H_i(\Delta,\Gamma)$) for the $i$th reduced simplicial homology group of a simplicial complex $\Delta$ (resp.\ the pair $(\Delta,\Gamma)$).
The following formula for Hilbert series of local cohomology modules is known as Hochster's formula for local cohomology.
See \cite[II, 4.1 Theorem]{Stanley} (and \cite[Theorem 1.8]{AdiSan} for the relative version).

\begin{lemma}
    \label{lem:SRlocalcohomology}
Let $\Delta \supset \Gamma$ be simplicial complexes on $[\ell]$.
Then
\[
\Hilb\big(H^i_{\mideal_S}(I_\Gamma/I_\Delta)\big)=\sum_{F \in \Delta} \big(
\dim_\K \widetilde H_{i-1-\#F}(\lk_\Delta(F),\lk_\Gamma(F)) \big) \frac {t^{-\#F}} {(1-t)^{-\#F}}.
\]
\end{lemma}

For a finitely generated graded $S$-module $M$, the number
\[\beta_{i,j}(M)=\dim_\K \Tor_i^S(M,\K)_j\]
is called the $(i,j)$th {\bf graded Betti number} of $M$,
and the number
\[\reg(M)=\max\{j \mid \beta_{i,i+j}(M) \ne 0 \text{ for some }j\}\]
is called the {\bf (Castelnuovo--Mumford) regularity} of $M$.
The next formula due to Hochster gives a combinatorial formula for graded Betti numbers and regularity of Stanley--Reisner ideals\footnote{The formula can be generalized to Stanley--Reisner modules but we do not need it in this paper.}.

\begin{lemma}
    \label{lem:SRbetti}
    For a squarefree monomial ideal $I=I_\Delta$ in $S$,
    one has
\[    \beta_{i,i+j}(I_\Delta)=\sum_{W \subset [n],\ \! \#W=i+j} \dim_\K \widetilde H_{j-2}(\Delta[W]).
\]
\end{lemma}

\begin{cor}
    \label{cor:SRreg}
 For a squarefree monomial ideal $I=I_\Delta$ in $S$,
    one has
\[ \reg(I_\Delta) \leq k \ \Leftrightarrow \ \widetilde H_j\big(\Delta[W]\big)=0 \text{ for all }W \subset [\ell] \text{ and } j\geq k-1.\]
\end{cor}

The formula in
Lemma \ref{lem:SRbetti} is called Hochster's formula for graded Betti numbers.

\subsection{Cover ideals and edge ideals}
We next introduce cover ideals and edge ideals of graphs
and discuss their basic properties.
We will later see that these ideals have a close connection to modules of logarithmic $1$-forms of graphic arrangements.

We first recall some terminology on graph theory.
Recall that a (finite simple) graph $G=(V(G),E(G))$ is a pair of a finite set $V(G)$ and a family $E(G)$ of 2-element subsets of $V(G).$
An element of $V(G)$ is called a {\bf vertex} of $G$ and an element of $E(G)$ is called an {\bf edge} of $G$.
We say that $G$ is the empty graph if $E(G) = \varnothing.$
The {\bf neighbor} of a vertex $v \in V(G)$ in $G$ is the set $N_G(v)=\{u \in V(G) \mid \{u,v\} \in E(G)\}$.
The {\bf complementary graph} $G^c$ of a graph $G$ is the graph whose vertex set is $V(G)$ and whose edges are non-edges of $G$.
Thus $E(G^c)=\{\{u,v\} \subset V(G) \mid \{u,v\} \not \in E(G)\}$.
For a graph $G$ and $W \subset V(G)$,
the {\bf induced subgraph} $G[W]$ of $G$ on $W$ is the graph whose vertex set is $W$ and whose edge set is
$\{ \{u,v\} \in E(G) \mid u,v \in W\}$.
We note that $(G[W])^c=G^c[W]$.

For a graph $G$ with the vertex set $[\ell]$,
its {\bf cover ideal} $J(G)$ is the ideal of $S$ defined by
\[J(G)=\bigcap_{\{u,v\} \in E(G)} (x_u,x_v)\]
and its {\bf edge ideal} $I(G)$ is the ideal of $S$ defined by
\[I(G)= (x_ux_v \mid \{u,v\} \in E(G)).\]
There is a concrete description of generators of cover ideals.
We say that $F \subset V(G)$ is a {\bf vertex cover} of a graph $G$ if $\{u,v\} \cap F \ne \varnothing $ for every $\{u,v\} \in E(G)$.
Also, we say that $F \subset V(G)$ is a {\bf clique} of a graph $G$ if for all distinct $u,v \in F$ one has $\{u,v\} \in E(G)$.
It is not hard to see that $F$ is a vertex cover of $G$ if and only if 
$V(G)\setminus F$ is a clique of $G^c$ (see \cite[Lemma 1.17]{Va}).
The following fact easily follows from the definition of cover ideals (see \cite[Lemma 1.20]{Va}).

\begin{lemma}
    \label{lem:genscoverideal}
    With the same notation as above, the set of squarefree monomials contained in $J(G)$ is 
    \[\{x^F\mid F \mbox{ is a vertex cover of $G$}\}
    =\{x^{[\ell]\setminus F}\mid F \mbox{ is a clique of }G^c\}.\]
    In particular, the above set generates $J(G).$
\end{lemma}

Let $\Delta(G)$ be the set of all cliques of a graph $G$.
This set $\Delta(G)$ forms a simplicial complex, which is called the {\bf clique complex} of $G$.
It is known that the edge ideal of $G$ is nothing but the Stanley--Reisner ideal of $\Delta(G^c)$, that is,
\begin{align}
\label{edgeclique}
I_{\Delta(G^c)}=I(G).
\end{align}
See e.g., \cite[Lemma 1.32]{Va}.
Another basic result on edge ideals is the following result due to Fr\"oberg \cite{Fr}.
A graph $G$ is said to be {\bf chordal} if $G$ contains no induced  cycle of length $\geq 4$.

\begin{theorem}[Fr\"oberg]
\label{froberg}
Let $G$ be a non-empty graph with the vertex set $[\ell]$. Then one has $\reg(I(G))=2$ if and only if $G^c$ is a chordal graph.
\end{theorem}

\subsection{Alexander duality}
There is a duality between cover ideals and edge ideals called the Alexander duality for simplicial complexes.
For a simplicial complex $\Delta$ on $[\ell]$,
its Alexander dual $\Delta^*$ is the simplicial complex
\[\Delta^*=\{[\ell]\setminus F \mid F \subset [\ell],\ F \not \in \Delta\}.\]
Note that $(\Delta^*)^*=\Delta$.
We say that a squarefree monomial ideal $I$ is the \textbf{Alexander dual} of a squarefree monomial ideal $J=I_\Gamma$ if $I=I_{\Gamma^*}$.
The next fact (see \cite[Corollary 1.5.5]{HH}) shows that cover ideals are Alexander duals of edge ideals.

\begin{lemma}
    \label{lem:SRdual}
    For a squarefree monomial ideal $I=I_\Delta$ in $S$,
    one has
    \[I_{\Delta^*}=\bigcap_{F \subset [\ell],\ \! x^F\in I_\Delta}(x_i \mid i \in F).\]
\end{lemma}

Following nice relations between $I_\Delta$ and $I_{\Delta^*}$ were proved by Eagon--Reiner \cite{ER98} and Terai \cite{Terai}.

\begin{lemma}[Eagon--Reiner]
\label{lem:SReagonReiner}
Let $I=I_\Delta$ be a squarefree monomial ideal in $S$ generated in degree $k$. Then
$\reg(I_\Delta)=k$ $\Leftrightarrow$ $S/I_{\Delta^*}$ is Cohen--Macaulay.
\end{lemma}

\begin{lemma}[Terai]
\label{lem:SRTerai}
For a squarefree monomial ideal $I=I_\Delta$ in $S$,
one has $\reg I_{\Delta}=\pd_S(S/I_{\Delta^*})$.
\end{lemma}

Combining Fr\"oberg's theorem and Eagon--Reiner theorem,
we get the following criterion of the Cohen--Macaulay property of cover ideals.

\begin{cor}
    \label{cor:CMcover}
    Let $G$ be a non-empty graph with the vertex set $[\ell]$.
    Then
    $S/J(G)$ is Cohen--Macaulay if and only if $G^c$ is a chordal graph.
\end{cor}

\subsection{Possible connection to the freeness of $\A^H$}

We close this section with a possible application of Corollary \ref{cor:CMcover} to the freeness of $\A^H$.
We explained in \S 3.3 that if $\A$ is a free arrangement and $H \in \A$ then $\A^H$ is free if and only if either the residue ideal $J_\A^H$ is trivial or a height two Cohen--Macaulay ideal.
In general, proving a Cohen--Macaulay property of an ideal is not easy, but if $J_\A^H$ happens to be a squarefree monomial ideal then one can use Corollary \ref{cor:CMcover} to check the height two Cohen--Macaulay property of $J_\A^H$.
The proposition below gives one case where we can apply this idea.
We say that an arrangement $\A=\{H_{\alpha_1},\dots,H_{\alpha_m}\}$ is Boolean if $\alpha_1,\dots,\alpha_m$ are linearly independent.

\begin{prop}
\label{7.11}
Let $\A$ be a free arrangement in $V$ and $H \in \A$.
Suppose that 
\begin{itemize}
    \item $\A_{\mathrm{core}}^H=\{H_{y_1},\dots,H_{y_m}\},$ where $y_1,\dots,y_m \in H^*$, is a Boolean arrangement,
    \item $\mathrm{mult}^\A(H_{y_k}) \leq 2$ for $k=1,2,\dots,m$.
\end{itemize}
Then $\A^H$ is free if and only if $J_\A^H=\overline S$ or $J_\A^H$ is the cover ideal of the complement of a chordal graph in variables $y_1,\dots,y_m$.
\end{prop}

\begin{proof}
Since a cover ideal of $G$ is Cohen--Macaulay if and only if $G^c$ is chordal, the if part follows from Proposition \ref{prop:4.12}.
We prove the only if part.
By Corollary \ref{cor:assprimecore}
every minimal prime of $J_\A^H$ is generated by subsets of $y_1,\dots,y_m$.
Then, since $y_1,\dots,y_m$ are linearly independent, by Theorem \ref{5.12}, if $\A^H$ is free then $J_\A^H$ must be a Cohen--Macaulay cover ideal in variables $y_1,\dots,y_m$ unless $J_\A^H=\overline S$.
\end{proof}

\begin{example}
\label{ex:DiPasquale}
Let $\ell \geq 3$.
Recently, DiPasquale \cite[Proposition 6.12]{Di} proved that the arrangement $\A$ in $\R^{\ell+1}$ defined by
\begin{align}
\label{dipasqual}
\textstyle Q(\A)=x_0 \left(\prod_{k=1}^\ell (x_k^2-x_0^2) \right)
(x_1+x_\ell)\left( \prod_{k=1}^{\ell-1} (x_k-x_{k+1})\right)
\end{align} 
is free but its restriction to $H=H_{x_0}$ is not free.
Let us explain the non-freeness of $\A^H$ assuming the freeness of $\A$ for this arrangement using the residue ideal $J_\A^H$.

Let $S=\K[x_0,x_1,\dots,x_\ell]$ and $y_k=\overline x_k \in \overline S=S/(x_0)$ for $k=1,\dots,\ell$.
For this arrangement $\A$, one can see that
\[\A_{\mathrm{core}}^H=\{H\cap H_{x_k-x_0}=H_{y_k} \mid k=1,2,\dots,\ell\}\]
is a Boolean arrangement and $\mathrm{mult}^\A(L)\leq 2$ for any $L \in \A^H$.
Thus we can apply Proposition \ref{7.11}.
Let
\[X_{i,j}=\{\bm x \in \R^{\ell+1} \mid x_0=x_i=x_j=0\}
\ \ \text{for }1 \leq i < j \leq \ell.\]
Then $J_\A^H$ must have a primary decomposition of the form
\[
J_\A^H=
\left(  \bigcap_{1 \leq i <j \leq n,\ \! X_{i,j} \in \NFT(\A,H)} J_{\A_{X_{i,j}}}^H
\right)
\bigcap
\left(
\bigcap_{P \in \mathrm{Ass}(\overline S/J_{\A}^H),\ \! \mathrm{height}(P)\geq 3} Q_P
\right),
\]
where $Q_P$ is a $P$-primary ideal.
Let $C_\ell$ be the cycle with edges $\{1,2\},\dots,\{\ell-1,\ell\},\{1,\ell\}$.
It is not hard to see that $X_{i,j}\in \NFT(\A,H)$
if and only if $\{i,j\}$ is not an edge of $C_\ell$.
(Indeed, the arrangement defined by $(x_i^2-x_0^2)(x_j^2-x_0^2)=0$ is not free
while the arrangement defined by $(x_i^2-x_0^2)(x_j^2-x_0^2)(x_i\pm x_j)=0$ is free.)
Hence
\begin{align}
\label{dipasqualeexample}
J_\A^H=
\left(  \bigcap_{\{i,j\} \not \in E(C_\ell)} (y_i,y_j)
\right)
\bigcap
\left(
\bigcap_{P \in \mathrm{Ass}(\overline S/J_{\A}^H),\ \! \mathrm{height}(P)\geq 3} Q_P
\right).
\end{align}
If $J_\A^H$ is a height two Cohen--Macaulay ideal,
then it cannot have an associated prime of height $\geq 3$ and the height two part of \eqref{dipasqualeexample} must be the cover ideal of a complement of a chordal graph.
If $\ell\geq 4$ then the latter condition is clearly not satisfied since the height two part of \eqref{dipasqualeexample} is the cover ideal of a complement of a cycle of length $\geq 4$.
On the other hand, if $\ell=3$ then $\A$ is linearly isomorphic to the arrangement in Example \ref{exdim4} by the correspondence $x\to x_0, y\to x_1, z \to -x_2,w \to x_3$,
so $J_\A^H=(y_1,y_2,y_3)$ and it has height $3$.
In both cases, $J_\A^H$ is not a height two Cohen--Macaulay ideal so $\A^H$ is not free.
\end{example}

\section{Residue ideals of graphic arrangements}

In this section, we study modules of 
logarithmic differential $1$-forms of graphic arrangements using residue ideals.
Since this section is relatively long, let us first give a summary of the results of this section.
\begin{itemize}
    \item In \S 7.1 and \S 7.2, we give a concrete presentation of the residue ideal $J_{\A_G}^H$ as well as a generating set of $\Omega^1(\A_G)$.
    \item In \S 7.3, we prove that the module $\Omega^1_R(\A_G)$ is isomorphic to $J(G^c)/(x_1\cdots x_\ell)$.
    \item In \S 7.4, we discuss local cohomology modules of $\Omega^1(\A_G)$ and give a closed formula for their Hilbert series. This in particular gives a closed formula for the Hilbert series of $D(\A_G)$.
    \item In \S 7.5, we show that the projective dimension of $\Omega^1(\A_G)$ coincides with the regularity of the edge ideal $I(G^c)$.
\end{itemize}

\subsection{Residue ideals of graphic arrangements}
We first discuss residue ideals of graphic arrangements.
Let $G$ be a graph with the vertex set $[\ell]=\{1,2,\dots,\ell\}$.
Recall that the {\bf graphic arrangement $\A_G$} of $G$ is the arrangement in $\K^\ell$ defined by
\[
\A_G=\{H_{x_u-x_v} \mid \{u,v\} \in E(G)\}.
\]
We now state our first result.
Let $H=H_{x_i-x_j} \in \A_G$,
\[
W=N_G(i) \cap N_G(j)=\{ u \in V(G) \mid \{u,i\} \in E(G) \mbox{ and } \{u,j\} \in E(G)\}
\]
and let $y_u=\overline{x_u-x_i} \in S/(\alpha_H)=S/(x_i-x_j)$ for $u \in W$.
Then it is easy to see that
\[(\A_G)^H_{\mathrm{core}}=\{ H \cap H_{x_u-x_i} \mid u \in W\}=\{H_{y_u} \mid u \in W\}.\]
From this computation and Proposition \ref{prop:2.10},
the residue ideal $J_{\A_G}^H$ must be generated by polynomials in $\{y_u\mid u \in W\}$.
This residue ideal actually has the following concrete presentation.

\begin{theorem}
    \label{thm:RDforGraph}
Let $G$ be a graph with the vertex set $[\ell]$ and $H=H_{x_i-x_j} \in \A_G$.
Let $W=N_G(i) \cap N_G(j)$. Then
\[
J_{\A_G}^H= \bigcap_{\{u,v\}  \subset W,\ \{u,v\} \not \in  E(G)} (y_u,y_v)
=\bigcap_{\{u,v\} \in E(G^c[W])}(y_u,y_v).
\]
\end{theorem}

As a byproduct of the proof of Theorem \ref{thm:RDforGraph},
we also find a generating set of $\Omega ^1(\A_G).$
Recall that for a subset $F\ne \varnothing$ of $[\ell]$, we define the element $\gamma_F \in \frac{1}{Q(\A)}\Omega^1_V$ by
\[
\gamma_F= \sum_{u \in F} \frac 1 {\prod_{v \in F,v \ne u} (x_v-x_u)} dx_u,
\]
where we consider that $\gamma_{\{p\}}=dx_p$ for each $p \in [\ell]$.

\begin{theorem}
    \label{thm:GensOmega}
Let $G$ be a graph with the vertex set $[\ell]$. Then the following set generates $\Omega^1(\A_G)$ as an $S$-module:
\[
\{\gamma_F \mid F \mbox{ is a clique of $G$ with $F \ne \varnothing$} 
\}.\]
\end{theorem}

\subsection{Proof of Theorems \ref{thm:RDforGraph} and \ref{thm:GensOmega}}

In this subsection,
we prove Theorems \ref{thm:RDforGraph} and \ref{thm:GensOmega}.
We recall the following freeness criterion for graphic arrangements due to Stanley (see \cite[Theorem 3.3]{ER94}).

\begin{theorem}[Stanley]
\label{stanley}
A graphic arrangement $\A_G$ is free if and only if $G$ is a chordal graph.
\end{theorem}

We also need the following technical statement on $\gamma_F$.

\begin{lemma}
\label{lem:gensOmega}
Let $G$ be a graph with the vertex set $[\ell]$ and let $ \varnothing \ne F \subset [\ell]$.
\begin{itemize}
\item[(1)]
If $|F| \geq 2$, then for any $p \in F$ one has
\[
\gamma_F=
\sum_{u \in F, u \not = p} \frac 1 {\prod_{v \in F,v \ne u} (x_v-x_u)} (dx_u-dx_p).\]
\item[(2)]
If $F$ is a clique of $G$, then $\gamma_F \in \Omega^1(\A_G)$.
\end{itemize}
\end{lemma}

\begin{proof}
(1) is an immediate consequence of the equation
\[
\sum_{i=1}^\ell \frac 1 {\prod_{1 \leq j\leq \ell, j \ne i} (x_j-x_i)} =0,
\]
which holds since the LHS is the product of the first row of the Vandermonde matrix and the last column of the inverse of the Vandermonde matrix.
\vspace{4pt}

We prove (ii).
Let $\{p,q\} \in E(G)$.
What we must prove is $Q(\A_G) \gamma_F \wedge (dx_p-dx_q) \in (x_p-x_q) \Omega^2_V$.
If $\{p,q\} \not \subset F$,
then $Q(\A_G) \gamma_F$ is already contained in $(x_p-x_q) \Omega^1_V$ so the claim is obvious.
Thus we may assume $F=\{1,2,\dots,m\}$ with $m \geq 2$ and $(p,q)=(1,2)$.
Since $Q(\A_G) \frac 1 {\prod_{v \in F,v \ne u} (x_v-x_u)}$ is divisible by $x_1-x_2$ when $u \not \in \{1,2\}$,
it suffices to prove that
\vspace{-6pt}

{\scriptsize
\begin{align}
    \label{eq:6-1}
    Q(\A_G) \left( \frac 1 {(x_2\!-\!x_1)(x_3\!-\!x_1)\cdots (x_m\!-\!x_1)}dx_1 + \frac 1 {(x_1\!-\!x_2)(x_3\!-\!x_2)\cdots(x_m\!-\!x_2)}dx_2\right)\wedge (dx_1-dx_2)
\end{align}
}

\noindent
is contained in $(x_1-x_2)\Omega_V^2$.
The element \eqref{eq:6-1} can be written as
\[
    \frac {Q(\A_G)}{(x_2-x_1)} \left(- \frac 1 {(x_3-x_1)\cdots (x_m-x_1)} + \frac 1 {(x_3-x_2)\cdots(x_m-x_2)}\right) dx_1 \wedge dx_2.
\]
Since we get zero if we substitute $x_2=x_1$ to the coefficient polynomial of $dx_1\wedge dx_2$ in the above element,
it follows that the element \eqref{eq:6-1} is  contained in $(x_2-x_1)\Omega_V^2$ as desired.
\end{proof}

\begin{proof}[Proof of Theorem \ref{thm:RDforGraph}]
Recall that $H=H_{x_i-x_j}$ and $W=N_G(i) \cap N_G(j)$.
Let
\[J= \bigcap_{\{u,v\} \subset W,\ \{u,v\} \not \in E(G)} (y_u,y_v)=\bigcap_{\{u,v\} \in E(G^c[W])} (y_u,y_v).\]
We first prove $J_{\A_G}^H \subset J$.
By Theorem \ref{thm:4.1} we have $J_{\A_G}^H \subset \bigcap_{X \in \Xi(\A_G,H)} P_X$.
Let $\{u,v\} \subset W$
with $\{u,v \} \not \in E(G)$.
Consider the subspace 
\[
X=\{(x_1,\dots,x_\ell) \in \mathbb K^\ell \mid x_i=x_j=x_u=x_v\}.
\]
Since $P_X=(y_u,y_v)$, to prove the desired inclusion, it suffices to prove that $X$ is contained in $\Xi(\A_G,H)=\{X \in L(\A_G^H)\mid J_{(\A_G)_X}^H \ne \overline S\}$.
Observe that $(\A_G)_X$ is the graphic arrangement of the following graph $G_0$ in Figure \ref{fig}.
By Stanley's criterion (Theorem \ref{stanley})
the arrangement $(\A_G)_X$ is free but $(\A_G)_X \setminus H$ is not free,
so by Theorem \ref{thm:3.4}(1) we get $X \in \Xi(\A,H)$ as desired.
\begin{figure}[h]
    \centering
   \includegraphics[width=0.15\linewidth]{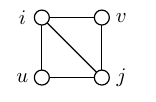}
    \caption{The graph $G_0$.}
    \label{fig}
\end{figure}

We next prove $J_{\A_G}^H \supset J$.
Let $F$ be a clique of $G[W]$.
By Lemma \ref{lem:genscoverideal},
to prove the desired inclusion,
it suffices to prove that $\prod_{p \in W \setminus F} y_p \in J_{\A_G}^H$.
Let $E=F \cup \{i,j\}$.
Then $E$ is a clique of $G$
and we have $\gamma_E \in \Omega ^1(\A_G)$ by Lemma \ref{lem:gensOmega}.
Using the fact that $Q(\A_G)$ and $Q(\A_G^H)$ can be written as
\[
Q(\A_G)= (x_j-x_i) \left(\prod_{p \in W} (x_p-x_i)(x_p-x_j)\right) \cdot g\]
and
\[
Q(\A^H_G)= \left(\prod_{p \in W} \overline{(x_p-x_i)}\right) \cdot \overline g\]
for some polynomial $g$,
we have
\begin{align}
    \label{66}
\mathrm{res}^H(\gamma_E)=
\frac 1 {Q(\A_G^H)} \overline{ \left(Q(\A_G) \frac 1 {(x_j-x_i)\prod_{v \in F}(x_v-x_i)} \right)}
=\prod_{p \in W \setminus F}y_p,
\end{align}
where we use Lemma \ref{lem:gensOmega}(1) for the first equation.
Hence $\prod_{p \in W \setminus F} y_p \in J_{\A_G}^H$ as desired.
\end{proof}

\begin{rem}
The proof of Theorem \ref{thm:RDforGraph} says that the images of $\gamma_F$ under the map $\mathrm{res}^H$ for all (maximal) cliques $F$ containing $\{i,j\}$ generates $J_{\A_G}^H$.
\label{remgenerate}
\end{rem}

\begin{proof}[Proof of Theorem \ref{thm:GensOmega}]
We prove that $\{\gamma_F \mid F \mbox{ is a clique of $G$ with $F \ne \varnothing $}\}$ generates $\Omega^1(\A_G)$ using induction on $\# \A_G$.
The statement is clear when $G$ is the empty graph.
Assume that $\A_G \ne \varnothing$ and take $H=H_{x_i-x_j} \in \A_G$.
Let $G'$ be the graph obtained from $G$ by deleting the edge $\{i,j\}$.
Consider the exact sequence which defined the residue ideal:
\[
0 \longrightarrow \Omega^1(\A_{G'})=\Omega^1(\A_G \setminus H)
\lhook\joinrel\longrightarrow \Omega^1(\A_G)
\stackrel{\mathrm{res}^H} \longrightarrow
J_{\A_G}^H \longrightarrow 0.
\]
The module $\Omega^1(\A_G)$ is generated by generators of $\Omega^1(\A_G \setminus H)$ together with elements of $\Omega^1(\A_G)$ whose images by the map $\mathrm{res}^H$ generate $J_{\A_G}^H$.
Since the former is generated by 
\[
\{\gamma_F \mid F \mbox{ is a clique of $G'$ with $F \ne \varnothing$} \}\]
by the induction hypothesis
and since $J_{\A_G}^H$ is generated by 
\[
\{\mathrm{res}^H(\gamma_F) \mid F \mbox{ is a clique of $G$ that contains } \{i,j\}\}\]
by Remark \ref{remgenerate},
we get the desired statement.
\end{proof}

Francisco, H\`a and Van Tuyl \cite{FHV10,FHV11} (see also \cite[\S 3]{Va}) proved that powers of cover ideals are related to odd holes and chromatic numbers of graphs.
Since cover ideals are residue ideals of graphic arrangements,
it might be interesting to study if powers of residue ideals have any combinatorial meaning.

\subsection{A direct relation to cover ideals}
We see in Theorem \ref{thm:RDforGraph} that the residue ideal of a graphic arrangement $\A_G$ is the cover ideal of a certain induced subgraph of $G^c$.
In this subsection,
we show that the module $\Omega^1(\A_G)$ itself directly relates to the cover ideal $J(G^c)$.
Recall that any graphic arrangement is not essential since the hyperplane $H_{x_i-x_j}$ contains $\K\cdot\bm 1$, the $1$-dimensional vector space spanned by the all $1$ vector $\bm 1=(1,\ldots,1)$.
Let
\[
R=
\mathrm{sym}\big(\big(V/(\K\cdot \bm 1)\big)^*\big)=\K[x_2-x_1,\dots,x_\ell-x_1]\]
and for a graph $G$ with the vertex set $[\ell]$, let
\[
\Omega_R^1(\A_G)=\Omega^1 (\A_G) \cap \left( \bigoplus_{i=0}^\ell R\cdot \frac 1 {Q(\A)}dx_i\right).
\]
The modules $\Omega_R^1(\A_G)$
and $\Omega^1(\A_G)$ have the same information.
Indeed,
by Lemma \ref{lem:2.5}, the module $\Omega^1(\A_G)$ is generated by elements of $\Omega^1_R(\A_G)$ and we have
\[\Omega^1(\A_G) \cong \Omega^1_R(\A_G)\otimes_R S.\]
We also note that $\Omega_R^1(\A_G)$ somehow sits between $\Omega^1(\A_G)$ and $\Omega^1(\A_G^{V/(\K\cdot \bm 1)})$ in the sense that
\[
\Omega^1_R(\A_G)=\Omega^1\left(\A_G^{V/(\K\cdot \bm 1)} \right) \oplus R \cdot dx_1.
\]

\begin{theorem}
    \label{thm:6.6}
    Let $G$ be a graph with the vertex set $[\ell]$.
    Then one has
    \[ \Omega^1_R(\A_G) \cong J(G^c)/(x_1x_2\cdots x_{\ell})\]
    as $R$-modules.
\end{theorem}

We need a few lemmas to prove Theorem \ref{thm:6.6}.
The next lemma follows from the definition of $\gamma_F$ by a straightforward computation.

\begin{lemma}
    \label{lem:6.8}
    For $F \subset [\ell]$ with $|F|\geq 2$ and $i,j \in F$, one has 
    \[(x_i-x_j) \gamma_F= \gamma_{F \setminus \{i\}} -\gamma_{F \setminus \{j\}}.\]
\end{lemma}

\begin{proof}
We may assume $F=\{1,2,\dots,m\}$, $i=1$ and $j=2$. Then we have
\begin{align*}
    &(x_1-x_2) \gamma_F\\
    &= \frac {-1} {\prod_{p=3}^m (x_p-x_1)}dx_1 + 
    \frac 1 {\prod_{p=3}^m (x_p-x_2)}dx_2 + \sum_{k=3}^m {\frac {(x_1-x_k)-(x_2-x_k)} {\prod_{1 \leq p \leq m, p \ne k} (x_p-x_k)}}dx_k\\
    &= \left( \frac {-1} {\prod_{p=3}^m (x_p-x_1)}dx_1
    - \sum_{k=3}^m {\frac {1} {\prod_{1 \leq p \leq m, p \ne 2,k} (x_p-x_k)}} dx_k\right) \\
    &\hspace{12pt} + 
    \left( \frac 1 {\prod_{p=3}^m (x_p-x_2)} dx_2 + \sum_{k=3}^m {\frac 1 {\prod_{1 \leq p \leq m, p \ne 1,k} (x_p-x_k)}} dx_k \right)\\
    &= -\gamma_{F \setminus \{2\}} + \gamma_{F \setminus \{1\}},
\end{align*}
as desired.
\end{proof}

Let $K_\ell$ be the complete graph with its vertex set $[\ell]$.
Thus $K_\ell$ is the graph with $V(K_\ell)=[\ell]$ and $E(K_\ell)=\{\{i,j\} \mid 1 \leq i < j \leq \ell \}$.
The next statement is an immediate consequence of Lemma \ref{lem:gensOmega}(2) and  Saito's criterion.

\begin{lemma}
\label{lem:6.9}
The module $\Omega^1_R(\A_{K_\ell})$ is a free $R$-module with a basis 
\begin{align}
\label{genscomplete}    
x\gamma_{\{1,2,\dots,\ell\}},
\gamma_{\{2,\dots,\ell\}},\dots,\gamma_{\{\ell-1,\ell\}},\gamma_{\{\ell\}}.
\end{align}
\end{lemma}

\begin{proof}
Since elements in \eqref{genscomplete} are contained in $\Omega_R^1(\A_{K_\ell})$ by Theorem \ref{thm:GensOmega},
it is enough to check two conditions in Theorem \ref{Saitocriterion}. 
Note that  elements in \eqref{genscomplete}
are $S$-independent since $\gamma_{\{i,i+1,\dots,\ell\}}$ 
is an $S$-linear combinations of $dx_i,dx_{i+1},\dots,dx_\ell$. Since $\deg 
(\gamma_{\{i,\ldots,\ell\}})=-\ell+i$ for $i=1,\ldots,\ell$, 
the sum of their degrees is $-\ell(\ell-1)/2=-|\A|$. So Theorem \ref{Saitocriterion} completes the proof.
\end{proof}

Consider the ring
\[A=S/(x_1x_2\cdots x_{\ell}).\]
This is a Cohen--Macaulay ring of Krull dimension $\ell-1$ with a system of parameters $z_2=x_2-x_1,\dots,z_\ell=x_\ell-x_1$.
Also, 
\[
S/(x_1x_2\cdots x_{\ell},z_2,\dots,z_\ell)=S/(x_1^{\ell},z_2,\dots,z_\ell)
\]
has a $\K$-basis $1,x_1,\dots,x_1^{\ell-2},x_1^{\ell-1}$
and
\[x_1 \cdots x_k=x_1^k \ \ \ \mbox{ mod }(z_2,\dots,z_\ell)\ \ \ \mbox{ for }k=1,2,\dots,\ell.\]
These facts tell us that $A$ is a free $R$-module with $R$-basis
\[
1,x_1,x_1x_2,\dots,x_1\cdots x_{\ell-1}.
\]
(see \cite[I,\ 5.10 Theorem]{Stanley}).
We see in Lemma \ref{lem:6.9} that $\Omega^1_R(\A_{K_\ell})$ is a free $R$-module with basis 
$\gamma_{[\ell]},\gamma_{[\ell]\setminus [1]},\dots,\gamma_{[\ell]\setminus [\ell-2]},\gamma_{[\ell]\setminus [\ell-1]}$.
Let
\[\Psi:\Omega^1_R(\A_{K_\ell}) \to A\]
be the isomorphism of free $R$-modules that sends each $\gamma_{[\ell]\setminus [k]}$ to $x^{[k]}=x_1\cdots x_k$, where $\Psi(\gamma_{[\ell]})=1$.
This map satisfies the following condition.

\begin{lemma}
\label{lem:6.10}
For any non-empty subset $F$ of $[\ell]$, one has $\Psi(\gamma_F)=x^{[\ell]\setminus F}.$
\end{lemma}

\begin{proof}
We prove the statement by downward induction on $\# F$.
The case $\# F=\ell$ follows from the definition of $\Psi.$
Assume $\#F =k < \ell$ and the statement holds for all subsets of size $\geq k+1$.
Clearly, we can take a sequence
\[
F=F_0,F_1,\dots,F_q=[\ell]\setminus [\ell-k]
\]
of $k$-subsets of $[\ell]$ such that $F_{i+1}=(F_i\setminus \{s_i\}) \cup \{t_i\}$ for some $s_i,t_i \in [\ell]$, 
$s_i \in F_i,\ t_i \not \in F_i$ 
for all $i=0,1,\dots,q-1$.
Set $E_i=F_i \cup F_{i+1} =(F_i\cap F_{i+1})\cup \{s_i,t_i\}$ for $i=0,1,\dots,q-1$.
By Lemma \ref{lem:6.8}
\begin{align*}
    \gamma_F&=(\gamma_{F_0}-\gamma_{F_1})+
    (\gamma_{F_1}-\gamma_{F_2})+ \cdots +(\gamma_{F_{q-1}}-\gamma_{F_q})+\gamma_{F_q}\\
    &=\sum_{i=0}^{q-1} \big((x_{t_i}-x_{s_i})\gamma_{E_i} \big) +\gamma_{[\ell]\setminus [\ell-k]}
\end{align*}
and by the induction hypothesis we have
\begin{align*}
    \Psi(\gamma_F) &= \left(\sum_{i=0}^{q-1} (x_{t_i}-x_{s_i})x^{[\ell]\setminus E_i} \right) + x^{[\ell-k]}\\
&= (x^{[\ell]\setminus F}-x^{[\ell]\setminus F_1})+ \cdots +(x^{[\ell]\setminus F_{q-1}}-x^{[\ell-k]})+x^{[\ell-k]}=x^{[\ell]\setminus F},
\end{align*}
as desired.
\end{proof}

\begin{proof}[Proof of Theorem \ref{thm:6.6}]
Consider the composition
\[\Psi_G :
\Omega_R^1( \A_G) 
\lhook\joinrel\longrightarrow
\Omega_R^1( \A_{K_\ell}) \stackrel \Psi \longrightarrow A.
\]
Since $\Psi$ is injective, to prove the theorem it suffices to prove 
\[\mathrm{Im}\ \! \Psi_G= J(G^c)/(x_1x_2\cdots x_{\ell}).\]
By Theorem \ref{thm:GensOmega} and Lemma \ref{lem:6.10}, $\mathrm{Im}\ \! \Psi_G$ is the $R$-submodule of $A$ generated by
\[
\mathcal G=\{x^{[\ell]\setminus F} \mid F \mbox{ is a clique of $G$ with $F \ne \varnothing$}\}.\]
On the other hand,
Lemma \ref{lem:genscoverideal} says that $J(G^c)/(x_1x_2\dots x_\ell)$ is the $S$-submodule of $A$ generated by $\mathcal G$.
Thus we have
\begin{align}
\label{eq:6-inclusion}
\mathrm{Im}\ \! \Psi_G \subset J(G^c)/(x_1x_2\cdots x_\ell).
\end{align}

It remains to prove the converse inclusion of \eqref{eq:6-inclusion}.
By Lemmas \ref{lem:SRHilb} and \ref{lem:genscoverideal},
we know that
\[
\bigcup_{x^K \in \mathcal G} \big\{ x^K \cdot m \mid m \text{ is a monomial in }\K[x_i \mid i \in K]\big\}
\]
is a $\K$-basis of $J(G^c)/(x_1\cdots x_\ell)$.
Thus, to prove the converse inclusion of \eqref{eq:6-inclusion} it suffices to prove $x^K m \in \mathrm{Im}\ \! \Psi_G$ for a fixed $x^K \in \mathcal G$ and a monomial $m \in \K[x_v \mid v \in K]$.
We prove this by induction on $\deg m$.
The assertion is obvious when $\deg m=0$ since $\mathrm{Im}\ \!\Psi_G$ contains $x^K$.
Suppose $\deg m>0$.
Take variables $x_s,x_t$ such that $x_s$ divides $m$ and $x_t$ does not divide $x^K$.
Write $m=x_s m'$.
Then
\[ x^Km = x^Kx_s m'= x^Km'(x_s-x_t)+x^{K\cup \{t\}} m'.\]
Since $x^Km' \in \mathrm{Im}\ \!\Psi_G$ by the induction hypothesis and $x_s-x_t \in R$ we have $x^Km'(x_s-x_t) \in \mathrm{Im}\ \!\Psi_G$.
Also, since $x^{K\cup \{t\}} \in \G$ or $x^{K\cup \{t\}}=0$ in $A$,
again by the induction hypothesis we have 
$x^{K\cup \{t\}} m'\in \mathrm{Im}\ \!\Psi_G$.
Hence $x^K m \in \mathrm{Im}\ \!\Psi_G$ as desired.
\end{proof}

\begin{rem}
The isomorphism from $\Omega_R^1(\A_G)$ to $J(G^c)/(x_1\cdots x_\ell)$ induces an $S$-module structure to $\Omega_R^1(\A_G)$. One can check that this module structure is given by
\[
\textstyle
x_i \cdot \big(\sum_{k=1}^\ell f_k (dx_k)\big)= \sum_{k=1}^\ell (x_i-x_k)f_k (dx_k).
\]
\end{rem}

As an immediate consequence of Theorem \ref{thm:6.6},
we have the following combinatorial formula of the Hilbert series of $\Omega^1(\A_G)$.
For a graded module $N=\bigoplus_{i\in\Z}N_i$, let $N(k)$ denote the module $N$ with the grading $N(k)_i=N_{i+k}$.

\begin{theorem}
\label{hilb:omega}
    Let $G$ be a graph with the vertex set $[\ell]$. Then
\[
\Hilb(\Omega^1(\A_G),t)=
\sum_{ \varnothing \ne F \in \Delta(G)}
\frac {t^{1-\#F}} {(1-t)^{\ell-\#F+1}}.
\]
\end{theorem}

\begin{proof}
Theorem \ref{thm:6.6} says that
$\Omega_R^1(\A_G)(-(\ell-1))$ is isomorphic to $J(G^c)/(x_1\dots x_\ell)$ as graded $R$-modules.
Since Lemma \ref{lem:genscoverideal} says $x^F \in J(G^c) $ if and only if $[\ell]\setminus F \in \Delta(G)$ for $F \subset [\ell]$,
it follows from Lemma \ref{lem:SRHilb} that
\[
\Hilb(J(G^c)/(x_1\cdots x_\ell),t)=
\sum_{ \varnothing \ne F \in \Delta(G)}
\frac {t^{\#([\ell]\setminus F)}} {(1-t)^{\#([\ell]\setminus F)}}.
\]
Since $\Hilb(\Omega^1(\A_G),t)=\frac 1 {1-t} \Hilb(\Omega^1_R(\A_G),t)$,
the desired formula follows from the above equation and the isomorphism
$\Omega^1_R(\A_G)(-(\ell-1))\cong J(G^c)/(x_1\cdots x_\ell)$.
\end{proof}

\subsection{local cohomology}
In this subsection,
we discuss local cohomology modules of $\Omega^1(\A_G)$.

Let $G$ be a graph with the vertex set $[\ell]$
and let $M(G)=J(G^c)/(x_1\cdots x_\ell)$.
The $R$-module $H_{\mideal_R}^{i}(M(G))$
and the $A$-module $H_{\mideal_RA}^{i}(M(G))$
are isomorphic as $R$-modules,
where $A=S/(x_1\cdots x_\ell)$,
and the $S$-module $H_{\mideal_S}^{i}(M(G))$ and the $A$-module $H_{\mideal_SA}^{i}(M(G))$ are isomorphic as $S$-modules
(see \cite[Proposition 2.14(2)]{Hu}).
Also,
the radical of the ideal $\mideal_RA$ equals to $\mideal_SA$,
so the $A$-module $H_{\mideal_RA}^{i}(M(G))$
and 
the $A$-module $H_{\mideal_S A}^{i}(M(G))$
are also isomorphic (see \cite[Proposition 2.13]{Hu}).
Thus we have
\[
H_{\mideal_R}^{i}\big(M(G)\big) \cong
H_{\mideal_RA}^{i}\big(M(G)\big) 
\cong
H_{\mideal_S}^{i}\big(M(G)\big)
\]
for all $i$ as graded $R$-modules.
Then, since $\Omega_R^1(\A_G) (\ell-1)\cong M(G)$ as graded $R$-modules,
it follows that
\begin{align}
\label{cohomologyiso}
H_{\mideal_R}^i\big(\Omega^1_R(\A_G)\big) (\ell-1)\cong 
H_{\mideal_R}^i\big(M(G)\big) 
\cong
H_{\mideal_S}^i \big(M(G)\big)
\end{align}
for all $i$ as graded $R$-modules.
Since 
\[
\Omega^1(\A_G) \cong \Omega^1_R(\A_G)\otimes_R S \cong \Omega_R^1(\A_G) \otimes_\K \K[x_1],\]
by the K\"unneth formula we have
\begin{align}
\label{formulacohomology2}    
H^i_{\mideal_S}\big(\Omega^1(\A_G)\big)
\cong
H^{i-1}_{\mideal_R}\big(\Omega^1_R(\A_G)\big) \otimes_\K H^1_{(x_1)}(\K[x_1])
\end{align}
for all $i$.
The isormorphisms \eqref{cohomologyiso} and \eqref{formulacohomology2} say that
knowing the Hilbert seris of $H^i_{\mideal_S}(\Omega^1(\A_G))$ is equivalent to knowing the Hilbert sereis of $H_{\mideal_S}^{i-1}(M(G))$.
By using Hochster's formula (Lemma \ref{lem:SRlocalcohomology}),
we can compute the Hilbert series of $H_{\mideal_S}^{i-1}(M(G))$.
For a simplicial complex $\Delta$,
we write $H_i(\Delta)=\widetilde H_i( \Delta,\{\varnothing\})$ for the ordinal $i$th homology of $\Delta$ with coefficients in $\K$.

\begin{lemma}
    \label{localcohomology}
Let $G$ be a graph with the vertex set $[\ell]$.
Then
\[
\Hilb\big(H_{\mideal_S}^{i-1}\big(M(G)\big) ,t\big)
= \sum_{\varnothing \ne W \subset [\ell]} 
\big(\dim_\K H_{\ell-i}\big(\Delta(G[W])\big)\big)
\frac {t^{-(\ell-\# W)}} {(1-t^{-1})^{\ell-\# W}}
\]
for all $i$.
\end{lemma}

\begin{proof}
Let $\Gamma$ and $\Sigma$ be simplicial complexes with
$I_\Gamma=J(G^c)$ and $I_\Sigma=(x_1\cdots x_\ell)$.
Thus $M(G)=I_\Gamma/I_\Sigma$ and $\Sigma=\{F \subset [\ell] \mid F \ne [\ell]\}$.
Also, the Alexander dual $\Gamma^*$ of $\Gamma$ is nothing but the flag complex $\Delta(G)$ since $I_{\Gamma^*}=I(G^c)=I_{\Delta(G)}$ by \eqref{edgeclique}.
By Lemma \ref{lem:SRlocalcohomology}
we have
\begin{align}
\label{eq7-10-1}    
H\big(H_{\mideal_S}^{i-1}\big(M(G)\big) ,t\big)= \sum_{ W \in \Sigma} \big(\dim_\K \widetilde H_{i-2-\#W} \big(\mathrm{lk}_\Sigma(W),\mathrm{lk}_\Gamma(W)\big) \big)\frac {t^{-\# W}} {(1-t^{-1})^{\# W}}.
\end{align}
On the other hand, since $\widetilde H_i(\Delta,\Delta')\cong \widetilde H_{\ell-2-i}((\Delta')^*,\Delta^*)$ for simplicial complexes $\Delta\supset\Delta'$ on $[\ell]$ (see \cite[Lemma 5.5.3]{BH}) we have
\begin{align}
\label{eq7-10-2}    
\widetilde H_{i-2-\#W} \big(\mathrm{lk}_\Sigma(W),\mathrm{lk}_\Gamma(W)\big)
\cong 
\widetilde H_{\ell-i}\big(\Gamma^{*}\big[[\ell]\setminus W\big],\Sigma^*\big[[\ell]\setminus W\big]\big),
\end{align}
where we use the fact that the Alexander dual of $\mathrm{lk}_\Delta(F)$ (on the vertex set $[\ell]\setminus F$) is nothing but $\Delta^*[[\ell]\setminus F]$.
Since $\Sigma^*=\{\varnothing\}$ and $\Gamma^*=\Delta(G)$,
the reduced homology $\widetilde H_{\ell-i}\big(\Gamma^{*}[W],\Sigma^*[W])$ is nothing but the ordinary homology $H_{\ell-i}\big(\Delta(G)[W])$.
Then the desired formula follows from \eqref{eq7-10-1} and \eqref{eq7-10-2}
\end{proof}

Since $H^1_{(x_1)}(\K[x_1]) \cong \K[x_1^{-1}](+1)$,
we have $\Hilb\big (H^1_{(x_1)}(\K[x_1]),t\big)=\frac {t^{-1}} {1-t^{-1}}$.
Then the above lemma together with \eqref{cohomologyiso} and \eqref{formulacohomology2} give the following formula for the Hilbert series of local cohomology modules of $\Omega^1(\A_G).$

\begin{cor}
    \label{cor:HilbLocalcohomologyOmega}
Let $G$ be a graph with the vertex set $[\ell]$. Then
\[
\Hilb\big( H^i_{\mideal_S}\big(\Omega^1(\A_G)\big),t\big)
= \sum_{\varnothing \ne W \subset [\ell]} 
\big(\dim_\K H_{\ell-i}\big(\Delta(G[W])\big)\big)
\frac {t^{-(2\ell-\# W)}} {(1-t^{-1})^{\ell-\# W+1}}
\]
for all $i$.
\end{cor}

The $i=\ell$ case of the above corollary gives a combinatorial formula for the Hilbert series of $D(\A_G)$,
since by the local duality (see \cite[Theorem 3.6.19]{BH}) we have
\[
D(\A_G)(-\ell) \cong \Hom_S\big(\Omega^1(\A_G),S(-\ell)\big)
\cong H_{\mideal_S}^{\ell}\big(\Omega^1(\A_G)\big)^\vee,
\]
where $M^\vee$ denotes the graded Matlis dual of $M$ (see \cite[\S 3.6]{BH}).
Let $c(G)=\dim_\K H_0(\Delta(G))$ be the number of connected components of a graph $G$.
Since $(M^\vee)_i\cong \Hom_\K(M_{-i},\K)$ for any graded $S$-module $M$,
by considering the case $i=\ell$ in Corollary \ref{cor:HilbLocalcohomologyOmega} we have
\begin{align}
    \label{HilbformD}
    \Hilb \big(D(\A_G),t\big)=\sum_{ \varnothing \ne  W \subset [\ell]} c(G[W]) \frac {t^{\ell-\#W }} {(1-t)^{\ell-\# W+1}}.
\end{align}
We slightly reformulate \eqref{HilbformD}.
For a graph $G$ with the vertex set $[\ell]$
we define
\[f_i(G)=\sum_{ W \subset [\ell], \#W=\ell-i} c(G[W])\]
for $i=0,1,\dots,\ell-1.$
Then \eqref{HilbformD} can be written as follows.

\begin{cor}
\label{cor:HilbD}
Let $G$ be a graph with the vertex set $[\ell]$.
Then
\[\Hilb\big(D(\A_G),t\big)= \sum_{i=0}^{\ell-1} f_i(G) \frac {t^i} {(1-t)^{i+1}}.\]
\end{cor}

\begin{example}
Let $G$ be a $4$-cycle
and $f_i=f_i(G)$ for $i=0,1,2,3$.
Then $f_0=1,f_1=4,f_2=8$ and $f_3=4$. Thus Corollary \ref{cor:HilbD} says
\begin{align*}
\Hilb(D(\A_G),t)&
=f_0\frac 1 {(1-t)}+f_1\frac t {(1-t)^2}+f_2\frac {t^2} {(1-t)^3}+f_3 \frac {t^3} {(1-t)^4}\\
&=\frac {(1-t)^3+4t(1-t)^2+8t^2(1-t)+4t^3} {(1-t)^4}\\
&=
\frac {1+t+3t^2-t^3} {(1-t)^4}.
\end{align*}
\end{example}

\begin{rem}
It is possible to prove Corollary \ref{cor:HilbD} inductively using the short exact sequence \eqref{exactTideal} in \S 8.4 without using local duality.
But we do not do that since it is not easy to guess this formula from that short exact sequence and an inductive proof may make the meaning of the formula unclear.
\end{rem}

The formula in Corollary \ref{cor:HilbD} has an algebraic meaning.
Indeed, inspired by this subsection,
we find an analogue of Theorem \ref{thm:6.6} for $D(\A_G)$ in \cite{AbeMurai2}.
There we show that the above $f_i(G)$ is the number of $(i-1)$-dimensional faces of a certain regular CW-complex when $G$ is connected.

\subsection{Projective dimension}
Corollary 
\ref{cor:HilbLocalcohomologyOmega} gives a formula of the projective dimension of $\Omega^1(\A_G)$ in terms of homologies of clique complexes.

\begin{cor}
\label{cor:formulaRegularityHomology}
If $G$ is a graph with the vertex set $[\ell]$, then
\[\pd_S\!\big(\Omega^1(\A_G)\big) \leq k
\ \ \Leftrightarrow \ \ H_j(\Delta(G[W]))= 0
\text{ for all $W \subset [\ell]$ and $j \geq k+1$}.\]
\end{cor}

\begin{proof}
By Corollary \ref{cor:HilbLocalcohomologyOmega},  the depth of $\Omega^1(\A_G)$ equals the smallest integer $k$ such that $H_{\ell-k}(\Delta(G[W]))\ne 0$ for some $\varnothing \ne W \subset [\ell]$. Then the desired statement follows from the Auslander--Buchsbaum formula.
\end{proof}

\begin{example}
Suppose that $\ell$ is even and let $G$ be the graph with $V(G)=[\ell]$ and
\begin{align*}
E(G)=E(K_\ell) \setminus \{\{1,2\},\{3,4\},\dots,\{\ell-1,\ell\}\}.
\end{align*}
Then $\Delta(G)$ is the boundary complex of the $\frac \ell 2$-dimensional cross polytope, so we have $H_{\frac \ell 2-1}(\Delta(G)) \cong \K$
and $H_k(\Delta(G[W]))=0$ for all $k> \frac \ell 2 -1$ and $W \subset [\ell]$. Thus
\[
\textstyle \pd_S(\Omega^1(G)) = \frac \ell 2 -1.
\]
\end{example}

On the projective dimension of $\Omega^1(\A_G)$,
we also have the following formula, 
which connects the projective dimension of $\Omega^1(\A_G)$
with the Castelnuovo--Mumford regularity of the edge ideal $I(G^c)$.

\begin{theorem}
    \label{thm:6.13}
For a graph $G \ne K_\ell$ with the vertex set $[\ell]$, one has
\[
\pd_S(\Omega^1(\A_G))= \reg(I(G^c))-2.
\]
\end{theorem}

\begin{proof}
The theorem easily follows from Corollaries \ref{cor:SRreg} and \ref{cor:formulaRegularityHomology} together with the fact that $I(G^c)=I_{\Delta(G)}$, but we give a more direct proof using $\Omega^1_R(\A_G) \cong M(G)$.

We first note that $\pd_S(\Omega^1(\A_G))=\pd_{R}(\Omega^1_R(\A_G))$ since $\Omega^1(\A_G)$ is generated by
elements of $\Omega^1_R(\A_G)$.
Since $G^c$ is not empty by the assumption we have $\pd_S(J(G^c)) \ne 0$.
Then the exact sequence
\[0 \longrightarrow (x_1 \cdots x_\ell)
\longrightarrow J(G^c) \longrightarrow M(G)=J(G^c)/(x_1\cdots x_\ell ) \longrightarrow
0 \]
proves
\[ \pd_S\!\big(J(G^c)\big)=\pd_S\!\big(M(G)\big).\]
Since $M(G) \cong \Omega_R^1(\A_G)$ is a finitely generated $R$-module,
we have
\[
\depth_S\!\big(M(G) \big)=\depth_R\!\big(M(G)\big)=\depth_R\!\big(\Omega^1_R(\A_G)\big).
\]
See e.g.\ \cite[Exercise 1.2.26]{BH} for the first equality.
Hence by the Auslander--Buchsbaum formula we have
\[ \pd_S\!\big(J(G^c)\big)=1+\pd_R\!\big(\Omega^1_R(\A_G)\big)
=1+\pd_S\!\big(\Omega^1(\A_G)\big).\]
Now the statement follows from Lemma \ref{lem:SRTerai}
which says that 
$\pd(J(G^c))$ is equal to $\reg(I(G^c))$ minus one.
\end{proof}

\begin{rem}
Theorem \ref{thm:6.13} shows that Theorem \ref{froberg} due to Fr\"oberg and Theorem \ref{stanley} due to Stanley are actually equivalent.
\end{rem}

The regularity of edge ideals is a well-studied topic in combinatorial commutative algebra.
By Theorem \ref{thm:6.13} we can directly apply results on $\reg(I(G))$ to $\pd_S(\Omega^1(\A_G))$.
Below we introduce one simple consequence.
We refer the readers to the survey \cite{BBH,MV} for more results on the regularity of edge ideals.

For any arrangement $\A$,
since $\Omega^1(\A)$ is reflexive, the projective dimension of $\Omega^1(\A)$ is bounded above by $\ell-2$. On the other hand, the regularity of the edge ideal of a graph with $\ell$ vertices is bounded above by $\frac \ell 2 +1$ (see \cite[Theorem 4.3]{BBH}).
Hence, for graphic arrangements, we get the following bound.

\begin{cor}
    \label{graphpd}
    Let $\ell \geq 2$. For a graph $G$ with the vertex set $[\ell]$, one has $\pd_S \Omega^1(\A_G) \leq \frac \ell 2 -1.$
\end{cor}

Note that Corollary \ref{graphpd} holds only for graphic arrangements, as the following examples show.

\begin{example}
(1)\,\,
Let $\A$ be an arrangement in $\R^4$ in Example \ref{exdim4} and $H=H_x$.
We see in Example \ref{ex:maximalPD} that $\Omega^1(\A\setminus H)$ has the maximal projective dimension.
Thus $\pd_S \Omega^1(\A\setminus H)=2 >\frac{4}{2}-1=1$. Clearly the inequality in Corollary \ref{graphpd} does not hold.

(2)\,\,
Let $\A$ be the Edelman--Reiner arrangement in Example \ref{ERexample} and let $H=H_{x_1+x_2+x_3+x_4+x_5}$.
We see in Example \ref{ex:maximalPD} that $\Omega^1(\A\setminus H)$ also has the maximal projective dimension, so $\pd_S \Omega^1(\A\setminus H)=3 > \frac{5}{2}-1$.
\label{exmaxpd}
\end{example}

We close this section with one question.
We show that Hochster's formula for local cohomology gives a formula for the Hilbert series of local cohomology modules of $\Omega^1(\A_G)$.
As we mentioned in \S 6,
there is also a combinatorial formula for the graded Betti numbers of $M(G)$ as $S$-modules,
known as Hochster's formula for graded Betti numbers (see Lemma \ref{lem:SRbetti}).
Unfortunately, this formula does not give a formula of the graded Betti numbers of the $S$-module $\Omega^1(\A_G)$ since they are equal to the graded Betti numbers of $\Omega^1_R(\A_G) \cong M(G)$ as $R$-modules not the graded Betti numbers of $M(G)$ as $S$-modules.
We pose the following problem.

\begin{problem}
\label{findBettiFormula}
Find a combinatorial formula for $\beta_{i,j}(\Omega^1(\A_G))$, equivalently a combinatorial formula for $\dim_\K \Tor^R_i(M(G),\K)_j$.
\end{problem}

\section{Concluding remarks, questions and related problems}

In this section, 
we collect several possible directions suggested by residue ideals.

\subsection*{Hilbert--Burch theorem}
Our original motivation is to develop a tool to study the freeness of  restrictions of a free arrangement.
By Proposition \ref{prop:4.12} this problem is reduced to the problem of the height two Cohen--Macaulay property of the corresponding residue ideal $J_\A^H$.
There is a well-known criterion for height two Cohen--Macaulay ideals known as Hilbert--Burch theorem \cite[Theorem 1.4.17]{BH}. It states that $I$ is a height two Cohen--Macaulay ideal of a polynomial ring $S$ if and only if $I$ has height $\geq 2$ and is generated by the maximal minors of $\ell \times (\ell+1)$ matrix for some $\ell.$
For example,
if we consider the ideal
\vspace{-10pt}

{\Small
\[
J_\A^L=\big( (x_2^2\!-\!x_3^2)(x_2^2\!+\!x_3^2\!-\!x_4^2\!-\!x_5^2),
(x_2^2\!-\!x_4^2)(x_2^2\!-\!x_3^2\!+\!x_4^2\!-\!x_5^2),
(x_2^2\!-\!x_5^2)(x_2^2\!-\!x_3^2\!-\!x_4^2\!+\!x_5^2) \big)
\]
}

\noindent
in Example \ref{ERexample}, which is a height two Cohen--Macaulay ideal,
then  it is the ideal of 2-minors of the following $2 \times 3$ matrix
\[
\begin{pmatrix}
    x_3^2-x_4^2 & x_2^2-x_5^2 & x_2^2-x_5^2 \\
    x_2^2-x_4^2 & x_3^2-x_5^2 & x_2^2-x_4^2 
\end{pmatrix}.
\]
Considering that the residue ideal $J_\A^H$ is an ideal of minors times $\frac 1 {Q(\A^H)}$ when $\A$ is free,
Hilbert--Burch theorem fits the situation of $J_\A^H$ but we could not find an interesting application of it.

\subsection*{Possible applications to codimension 2 subspace arrangements}
Theorem \ref{thm:RDforGraph} says that every (Cohen--Macaulay) codimension $2$ coordinate subspace arrangement can appear as the residue ideal of a (free) graphic arrangement.
Motivated by this fact we ask

\begin{question}
For any (Cohen--Macaulay) codimension $2$ subspace arrangement $\B$,
can we find a (free) arrangement $\A$ and $H \in \A$ such that $J_\A^H$ gives the defining ideal of $\B$? 
\end{question}

\begin{example}[Shi arrangement and Vandermonde determinant ideal]
We give one more example of a Cohen--Macaulay subspace arrangement that appears as a residue ideal.

Let $J \subset S=\K[x_1,\dots,x_\ell]$ be the ideal of the subspace arrangement consisting of all subspaces defined by the equations
\begin{align}
    \label{6-1}
x_i=x_j=x_k
\end{align}
as well as those defined by
\begin{align}
    \label{6-2}
x_i=x_j \mbox{ and }x_s=x_t.
\end{align}
Thus
\[J= \left(\bigcap_{\{i,j,k\} \subset [\ell]} (x_i-x_j,x_i-x_k) \right) \cap \left( \bigcap_{\{i,j\},\{s,t\} \subset [\ell] \atop \{i,j\} \cap \{s,t\} =\varnothing} (x_i-x_j,x_s-x_t)\right).\]
This ideal is a special case of the ideal studied by Lov\'asz \cite{Lo}.
Its minimal generators are given by maximal minors of $(\ell +1)\times \ell$ Vandermonde matrix and the ideal is known to be Cohen--Macaulay (see \cite{WY}).

We show that the ideal $J$ is a residue ideal of the Shi arrangement.
Let $\A$ be the (type A) Shi arrangement in $\K^{\ell+1}$. Thus $\A$ is the arrangement in $\K^{\ell+1}$ 
with
\[
Q(\A)= z \prod_{1 \leq i < j \leq \ell} (x_i-x_j)(x_i-x_j+z).
\]
Let $H=H_z$.
It is known that $\A$ and $\A^H$ are free but $\A\setminus H$ is not free. Thus $J_\A^H$ is a height two Cohen--Macaulay ideal. Moreover, the pair $(\A,H)$ satisfies the condition (\#) in Corollary \ref{5.12},
so $J_\A^H \subset S=\K[x_1,\dots,x_n] $ is the ideal of an $(\ell-2)$ dimensional subspace arrangement $\NFT_{3}(\A,H)$.
It is not hard to see that
\[
\NFT_{3}(\A,H)=\{X \in L(\A^H)\mid \dim X=\ell-2\}
\]
and it consists of all subspaces of the form \eqref{6-1} and \eqref{6-2}.
Thus $J_\A^H=J$.
\end{example}

\subsection*{Generators of $\Omega^1(\A)$}
Descriptions of $\Omega^1(\A_G)$ in Theorems \ref{thm:GensOmega} and \ref{thm:6.6} are unexpected applications of residue ideals for us.
We are not sure if there is a description of the module $\Omega^1(\A)$ like Theorem \ref{thm:6.6} outside of the class of graphic arrangements.
On the other hand,
we think that it might be possible to obtain generators of $\Omega^1(\A)$ like Theorem \ref{thm:GensOmega} for some other classes of arrangements using similar ideas.
Since graphic arrangements are subarrangements of Weyl arrangements of type $A$,
the following question would be natural.

\begin{question}
Is it possible to extend Theorem \ref{thm:GensOmega} to subarrangements of the Weyl arrangement of type B or type D?
\end{question}

We hope that such a study is useful to study freeness of subarrangements of Weyl arrangements of type B or type D.
A closely related problem would be the following.

\begin{problem}
Determine residue ideals of subarrangements of Weyl arrangements.
Also, determine when they are height two Cohen--Macaulay ideals.
\end{problem}

\subsection*{Vector fields}

We can also consider a dual version of residue ideals.
We briefly discuss it here.
For an arrangement $\A$ in $V$ and $H \in \A$, if we replace $\Omega^{\ell-1}(-)$ with $D^1(-)$ in the dual Euler sequence \eqref{3-2} we get the exact sequence
\begin{align}
    0 \longrightarrow D(\A) \lhook\joinrel\longrightarrow  D(\A\setminus H) 
    \stackrel{\rho_H^\A} \longrightarrow \overline S,
\end{align}
where the last map $\rho_H=\rho_H^{\A}$ is given by $\rho_H(\theta)=\overline {\theta\cdot \alpha_H}$ for $\theta \in D(\A \setminus H)$.
We define 
\[T_\A^H=\mathrm{Im}\ \! \rho_H^\A \subset \overline S.\]

\begin{example}
\label{ex:restrictionideal}
Let $\A=\A_G$ be the graphic arrangement of the cycle $G=C_\ell$ with $E(C_\ell)=\{\{1,2\},\{2,3\},\dots,\{\ell-1,\ell\},\{1,\ell\}\}$ and let $H=H_{x_1-x_\ell}$. In this case, $D(\A_G \setminus H)$ is a free module generated by $\sum_{i=1}^\ell \partial_{x_i}$ and
$\theta_k=(x_k-x_{k+1})(\sum_{i=1}^k \partial_{x_i})$ for $k=1,2,\dots,\ell-1$.
Thus we have
\[T_\A^H=(\rho_H^\A(\theta_1),\dots,\rho_H^\A(\theta_{\ell-1}))=(x_1-x_2,x_2-x_3,\dots,x_{\ell-1}-x_\ell).\]
\end{example}

Many of the basic statements proved for residue ideals $J_\A^H$ in this paper can be proved for $T_\A^H$  in the same way as for $J_\A^H$.
For example,
\begin{itemize}
    \item $T_\A^H=T^{H/X}_{\A^{V/X}} \overline S$ for any subspace $X$ of $X_\A=\bigcap_{L \in \A} L$.
    \item $T_{\A_X}^H \supset T^H_\A$ for any $X \in L(\A^H)$.
    \item if we define
    $\Lambda(\A,H)=\{ X \in L(\A^H) \mid T_\A^H \ne \overline S\}$
    then $\sqrt {T_\A^H} = \bigcap_{X \in \Lambda(\A,H)} P_X$.
    \item $T_\A^H=\bigcap_{P_X \in \Ass(\overline S/T_\A^H)} T_{\A_X}^H$.
\end{itemize}
The verification is left to the reader.

Like residue ideals, it is also possible to compute the ideals $T_\A^H$ for graphic arrangements.
In fact, the ideal $T_{\A_G}^H$ was implicitly computed in a recent work of M\"uhlherr \cite{Mu} who find a generating set of $D(\A_G)$.
Let us explain this.
Let $G$ be a graph on $[\ell]$ and $u,v \in [\ell]$ with $\{u,v\} \not \in G$.
The {\bf $u,v$-path} in $G$ is a sequence of distinct vertices $u=u_0,u_1,\dots,u_m,u_{m+1}=v$ with $\{u_i,u_{i+1}\} \in E(G)$ for all $i$.
Also, a {\bf $u,v$-separator} in $G$ is a subset $T=\{w_1,\dots,w_\ell\} \subset [\ell]\setminus\{u,v\}$ such that $u$ and $v$ are disconnected in $G[[\ell]\setminus T]$.
We define the {\bf $u,v$-separator ideal} $T_{u,v}(G)$ of $G$ as the monomial ideal
\[
T_{u,v}(G)=( x^T \mid T \text{ is a $u,v$-separator in $G$} ).
\]
It is not difficult to see that
\[
T_{u,v}(G)= \bigcap_{u,u_1,\dots,u_m,v \text{ is a }u,v\text{-}\text{path of }G}
(x_{u_1},\dots,x_{u_m}).
\]
The following result is implicitly proved in \cite{Mu}.

\begin{theorem}[M\"uhlherr]
Let $G$ be a graph with the vertex set $[\ell]$,
$\{u,v\} \in E(G)$
and let $G'$ be the ideal obtained from $G$ by removing the edge $\{u,v\}$. Let $y_i=x_i-x_u$ for all $i\ne u$.
Then
\begin{align}
    \label{764}
T_{\A_G}^{H_{x_u-x_v}}= (y_{t_1} \cdots y_{t_m} \mid \{t_1,\dots,t_m\} \text{ is a $u,v$-separator in $G'$}).
\end{align}
\end{theorem}

\begin{proof}
We sketch the proof since this essentially appears in the proof of \cite[Proposition 3.7]{Mu}.
To simplify notation let $\A=\A_G$ and $H=H_{x_u-x_v}$.
We first prove that the left hand side contains the right hand side in \eqref{764}.
Let $\{t_1,\dots,t_m\}$ be a $u,v$-separator of $G'$.
We prove $y_{t_1} \cdots y_{t_m} \in T^H_{\A}$.
Let $C$ be the set of vertices that belong to the same connected component of $G'[[\ell]\setminus T]$ as $u$.
Then the element
\[\theta= \sum_{i \in C} \prod_{t \in T} (x_t-x_i) \partial_{x_i} \]
is contained in $D(\A_{G'})$ (see \cite[Lemma 3.4]{Mu}), and we have
\[y_{t_1} \cdots y_{t_m}= \overline {\theta \cdot (x_u-x_v)} \in T_{\A}^H,\]
as desired.

We next prove that the RHS contains the LHS in \eqref{764}. To prove this it suffices to prove that, for any induced path $u,u_1,u_2,\dots,u_m,v$ in $G'$, one has
\[T^H_\A \subset (y_{u_1},y_{u_2},\dots,y_{u_m}).\]
Let $X$ be the subspace of $H$ defined by the equations $x_u=x_v=x_{u_1}=\cdots=x_{u_m}$.
Then the arrangement $\A_X$ is the graphic arrangement of the graph whose edge set is the cycle with vertices $u,u_1,\dots,u_m,v$.
Then by the computation given in Example \ref{ex:restrictionideal} we have
\[ T^H_\A \subset T^H_{\A_X}=(x_u-x_{u_1},x_{u_1}-x_{u_2},\dots,x_{u_m}-x_v)
=(y_{u_1},\dots,y_{u_m}),\]
as desired.
\end{proof}

Consider the exact sequence
\begin{align}
\label{exactTideal}
0 \longrightarrow D(\A_G) \lhook\joinrel\longrightarrow  D(\A_G\setminus H) 
    \longrightarrow T_{\A_G}^{H} \longrightarrow 0.
\end{align}
By the above theorem of M\"uhlherr,
the ideal $T_{\A_G}^H \subset \overline S$ is the $u,v$-separator ideal of $G'$,
so it might be possible to obtain interesting algebraic consequences on $D(\A_G)$ by studying algebraic properties of $u,v$-separator ideals.
In the study of squarefree monomial ideals,
it sometimes happens that the problem becomes easier if one works with Alexander duals.
So, it might be also interesting to study the Alexander dual of $T_{u,v}(G)$ which is the ideal
\[
(x_{u_1}\cdots x_{u_m} \mid u,u_1,\dots,u_m,v \text{ is a }u,v\text{-}\text{path in }G).
\]

\subsection*{MAT-free graphic arrangements}
MAT-free arrangements form a strict subclass of free arrangements with a very good property. Let us explain.

We say that an essential arrangement $\A$ in $\K^\ell$ is \textbf{MAT-free} if there is a partition 
$$
\A=\A_1 \uplus \A_2 \uplus \cdots \uplus \A_n
$$
such that 
\begin{itemize}
\item[(1)]
 $\codim \bigcap_{H \in \A_i} H=|\A_i|$ for $i=1,\ldots,n$,
\item[(2)]
$X_i \not \subset H$ for all $H \in \B_{i-1}:=\bigcup_{j=1}^{i-1} \A_j$ for $i=2,\ldots,n$, and 
\item[(3)]
$|\B_{i-1}|-|(\B_{i-1} \cup H)^H|=i-1$ for $H \in \A_i$ and $i=2,\ldots,n$.
\end{itemize}
Then the multiple addition theorem in \cite{ABCHT} shows that $\A$ is free with $\exp(\A)=(d_1,\ldots,d_\ell)$, where for $i_j:=|\A_j|$, 
$$
\exp(\A)=((0)^{\ell-i_1},(1)^{i_1-i_2},\ldots,(n-1)^{i_{n-1}-i_n},(n)^{i_n}).
$$
Here $(a)^b$ implies that the integer $a$ appears exactly $b$-times.

This was originally proved for ideal subarrangements in \cite{ABCHT}, and the concept of MAT-free arrangements was introduced and studied in \cite{CM}. What is interesting is that 
a MAT-free arrangement enables us to read the exponents by the dual partition of the height distribution, whose origin is nothing but in the root system. MAT-freeness is a generalization of Weyl arrangements to a general arrangement from this point of view. Surprisingly, in \cite{TranT}, it was shown that the MAT-free property for the graphic arrangement corresponds to the strongly chordal property of the graphs. So we can ask the following question.

\begin{problem}
Does the MAT-free property of $\A_G$, equivalently the strong chordality of $G$,
have any algebraic meaning for $J(G^c)$ (or for $J(G^c)/(x_1\cdots x_\ell)$)?
\end{problem}

It is known that clique complexes of strongly chordal graphs
are precisely the simplicial complexes called forests.
This fact gives a combinatorial criterion of the minimal monomial generators of $J(G^c)$ for strongly chordal graphs $G$.
See \cite[\S 4]{HHTZ}.

\subsection*{Multi-graphic arrangements}

A pair $(\A,m)$ of an arrangement $\A$ and 
a multiplicity $m:\A \rightarrow \Z_{\ge 0}$ is called a \textbf{multiarrangement}. To $(\A,m)$ we can define the logarithmic modules as follows:
\begin{eqnarray*}
D(\A,m):&=&\{\theta \in \Der S \mid \theta(\alpha_H) \in S \alpha_H^{m(H)} \ (\forall H \in \A)\},\\
\Omega^1(\A,m):&=& \{\omega \in 
\frac{1}{Q(\A,m)} \Omega^1_V \mid Q(\A,m) \omega \wedge d\alpha_H \in \alpha_H^{m(H)} \Omega^2_V\ (\forall H \in \A)\}.
\end{eqnarray*}
Here $Q(\A,m):=\prod_{H \in \A} \alpha_H^{m(H)}$. They are also dual $S$-modules, hence reflexive and their projective dimension is at most $\ell-2$. We can define, by using the $C$-sequence in 
\cite{A16}, the residue ideal $J_{(\A,m)}^H$ in exactly the same manner, and almost all results in this paper can be proved for $J_{(\A,m)}^H$. However, right now it is difficult to obtain the multi-version of the results in \S 7. For example, Theorem \ref{pdupper} is no more true for multi-case 
as follows.

\begin{prop}
$\pd_S \Omega^1(\A_G,m)$ can attain all the values from $0$ to $\ell-3$.    
\label{allattain}
\end{prop}

\begin{proof}
Let $\A:=\A_{K_\ell}$. 
Recall the isomorphism
$$
D(\A,m) \rightarrow \Omega^1(\A,2k-m)
$$
for a constant $k \in \Z_{>0}$ and $m:\A \rightarrow \{0,1\} $ shown in \cite{AY0}.
Thus 
$$
\pd_S \Omega^1(\A,2k-m)=\pd_S D(\A,m).
$$
Since the $r$-cycle $C_r$ makes $\pd_S D(\A_{C_r})=r-3$, we complete the proof.
\end{proof}

Thus for example, a multiarrangement $(\A,m)$ defined by 
$$
\displaystyle \frac{\prod_{1 \le i < j \le \ell}(x_i-x_j)^{2k}}{(x_1-x_\ell)\prod_{i=1}^{\ell-1}(x_i-x_{i+1})}=0
$$
for $k>0$ satisfies that 
$$
\pd_S \Omega^1(\A,m)=\ell-3
$$
for $\ell \ge 4$. On the other hand, by \cite{AY}, we can show the following.

\begin{prop}
\label{2k-m}
Let $\A:=\A_{K_\ell}$, $k>0$ and $m:\A \rightarrow \{0,1\}$. Then 
$\Omega^1(\A,2k+m)$ is generated by 
$$
\{\nabla_{I^*(\gamma_F)} \nabla_{D}^{k-1} dP_1\mid F\ \mbox{is a clique of}\ G\},
$$
where $D$ is K. Saito's primitive derivation, 
$P_1=x_1^2+\cdots+x_\ell^2$, 
$I$ is the $S_\ell$-invariant inner product and $I^*$ gives an identification between $(\Omega^1_V)_{(0)}$ and $(\Der S)_{(0)}$, 
and $\nabla$ is the affine connection. See \cite{Sa} for details. Also,
$$
\pd_S \Omega^1(\A,2k+m) \le \frac{\ell}{2}-1.
$$
\end{prop}

\begin{proof}
Recall the isomorphism
$$
\Omega^1(\A,m) \rightarrow \Omega^1(\A,2k+m)
$$
for a constant $k \in \Z_{>0}$ and $m:\A \rightarrow \{0,1\} $ shown in \cite{AY0} given by 
$$
\Omega^1(\A,m) \ni \omega \mapsto 
\nabla_{I^*(\omega) } \nabla_D^{k-1} dP_1.
$$
Thus 
$$
\pd_S \Omega^1(\A,2k+m)=\pd_S \Omega^1(\A,m).
$$
So our results in the previous sections complete the proof. 
\end{proof}

So the problem is for multiplicity $m$ on $\A=\A_{K_\ell}$ with 
$H,L \in \A$ such that $|m(L)-m(H)| \ge 2$. 

\begin{problem}
Can we describe the generators, projective dimension, and residue ideals for $(\A_{K_\ell},m)$ in general?
\end{problem}


\begin{thebibliography}{ABCHT}

\bibitem{A}
T. Abe, 
Roots of characteristic polynomials 
and intersection points of line arrangements. 
\textit{J. Singularities}, 
\textbf{8} (2014), 100--117.


\bibitem{A4}
T. Abe, 
Deletion theorem and combinatorics of hyperplane arrangements.
\textit{Math. Ann}.
\textbf{373} (2019), issue 1--2, 581--595. 


\bibitem{A5}
T. Abe, 
Plus-one generated and next to free arrangements of hyperplanes. 
\textit{Int. Math. Res. Not.} \textbf{2021} (2021), no. 12, 9233--9261.


\bibitem{A9}
T. Abe, 
Projective dimensions of hyperplane arrangements.
\textit{Trans. Amer. Math. Soc.} \textbf{377} (2024), no. 11, 7793--7827.



\bibitem{A14}
T. Abe,
Addition-deletion theorems for the Solomon-Terao polynomials and $B$-sequences of hyperplane arrangements.  
\textit{Math. Z.} \textbf{306} (2024), no. 2, 25. 

\bibitem{A16}
T. Abe,
Tame arrangements.
arXiv:2504.14902.


\bibitem{ABCHT}
T. Abe, M. Barakat, M. Cuntz, T. Hoge, H. Terao,
The freeness of ideal subarrangements of Weyl arrangements. 
\textit{J. Eur. Math. Soc}. 
\textbf{18} (2016), no. 6, 1339--1348.



\bibitem{AD}
T. Abe and G. Denham, 
Deletion-Restriction for Logarithmic Forms on Multiarrangements. 
\textit{Adv. applied Math}. \textbf{179} (2026).

\bibitem{AHMMS}
T. Abe, T. Horiguchi,  M. Masuda, S. Murai and T. Sato, 
Hessenberg varieties and hyperplane arrangements. 
\textit{J. Reine Angew. Math.} 
\textbf{764} (2020), 241--286.


\bibitem{AbeMurai2}
T. Abe and S. Murai,
Vector fields of graphic arrangements and face rings of simplicial posets,
arXiv:2608.06767.

\bibitem{AY0}
T. Abe and M. Yoshinaga,
Coxeter multiarrangements with quasi-constant
multiplicities.
\textit{J. Algebra} \textbf{322} (2009), 2839--2847.


\bibitem{AY}
T. Abe and M. Yoshinaga, 
Free arrangements and coefficients of characteristic polynomials. 
\textit{Math. Z}., 
\textbf{275} (2013), Issue 3, 911--919.


\bibitem{AdiSan}
K. Adiprasito and R. Sanyal,
Relative Stanley--Reisner theory and upper bound theorems for Minkowski sums.
\textit{Publ. Math. Inst. Hautes \'Etudes Sci.} \textbf{124} (2016), 99--163.


\bibitem{AM}
M. F. Atiyah and I. G. Macdonald,
\textit{Introduction to commutative algebra}.
Addison-Wesley Publishing Co., 1969. 


\bibitem{BBH}
A. Banerjee, S.K. Beyarslan, H. Huy T\`ai,
Regularity of edge ideals and their powers.
Springer Proc. Math. Stat., \textbf{277}
Springer, Cham, 2019, 17--52.

\bibitem{BH}
W. Bruns and J. Herzog, 
\textit{Cohen-Macaulay Rings 2nd ed}. 
Cambridge University Press. 

\bibitem{CM}
M. Cuntz and P. M\"ucksch, 
MAT-free reflection arrangements.
\textit{Electron. J. Combin}. \textbf{27} (2020) no. 1, Paper no. 1.28, 28 pages.

\bibitem{DSSWW}
G. Denham, H. Schenck, M. Schulze, M. Wakefield 
and U. Walther, 
Local cohomology of logarithmic forms.
\textit{Ann. Inst. Fourier} (Grenoble) 
\textbf{63} (2013), no. 3, 1177--1203.

\bibitem{Di}
M. DiPasquale,
A homological characterization for freeness of multi-arrangements.
\textit{Math. Ann.} {\bf 385} (2023), 745--786.

\bibitem{DW}
M. DiPasquale and M. Wakefield,
Free multiplicities on the moduli of $X_3$.
\textit{J. Pure Appl. Algebra} \textbf{222} (2018), no. 11, 3345--3359.

\bibitem{DST}
M. DiPasquale, J. Sidman and W. Traves,
Geometric aspects of the Jacobian of a hyperplane arrangement.
\textit{Int. Math. Res. Not. IMRN} \textbf{2025}, no. 13, Paper No. rnaf172, 33 pp.


\bibitem{ER98}
J.A. Eagon and V. Reiner,
Resolutions of Stanley-Reisner rings and Alexander duality.
\textit{J. Pure Appl. Algebra} \textbf{130} (1998), no. 3, 265--275.


\bibitem{ER}
P.H. Edelman and V. Reiner, 
A counterexample to Orlik's conjecture.
\textit{Proc. Amer. Math. Soc}., \textbf{118} (1993), 
927--929.

\bibitem{ER94}
P.H. Edelman and V. Reiner,
Free hyperplane arrangements between $A_{n-1}$ and $B_n$.
\textit{Math. Z.} \textbf{215} (1994), no. 3, 347--365.


\bibitem{Ei}
D. Eisenbud,
Commutative algebra: with a view toward algebraic geometry, Graduate Texts in Mathematics, vol.\ 150, Springer-Verlag, New York, 1995.

\bibitem{FHV10}
C.~A. Francisco, H.~T. H\`a{} and A. Van~Tuyl, Associated primes of monomial ideals and odd holes in graphs, \text{J. Algebraic Combin.} {\bf 32} (2010), no.~2, 287--301.

\bibitem{FHV11}
C. Francisco, H.~T. H\`a{} and A. Van~Tuyl, 
Colorings of hypergraphs, perfect graphs, and associated primes of powers of monomial ideals,
\textit{J. Algebra} \textbf{331} (2011), 224--242.

\bibitem{Fr}
R. Fr\"oberg,
On Stanley–Reisner Rings.
In: Topics in Algebra, Part 2 (Warsaw, 1988) Banach Center Publication, vol. 26, pp. 57–70. PWN, Warsaw (1990).


\bibitem{HH}
J. Herzog and T. Hibi,
Monomial Ideals, Graduate Texts in Mathematics, vol. 260, Springer--Verlag, 2011.


\bibitem{HHTZ}
J. Herzog, T. Hibi, N.V. Trung, and X. Zheng,
Standard graded vertex cover algebras, cycles and leaves, \textit{Trans. Amer. Math. Soc.} {\bf 360} (2008), no.~12, 6231--6249.



\bibitem{Hu}
C. Huneke, Lectures on local cohomology. In Interactions between homotopy theory and algebra, Contemp. Math., vol. 436, pp. 51–99, Appendix 1 by Amelia Taylor.


\bibitem{M2}
D.R. Grayson and M.E. Stillman. Macaulay2, a software system for research in algebraic geometry. Available at {\tt http://www.math.uiuc.edu/Macaulay2/}



\bibitem{Lo} L. Lov\'asz,
Stable sets and polynomials.
\textit{Discrete Math.} \textbf{124} (1994), 137--153.


\bibitem{Mu}
Leonie M\"uhlherr,
Separator-based derivations of graphic arrangements.
arXiv:2504.19893.



\bibitem{MS}
M. Musta\c{t}\u{a} and H. Schenck, 
The module of logarithmic $p$-forms of a locally free arrangement.
\textit{J. Algebra} \textbf{241} (2) (2001), 699--719.


\bibitem{MN}
J. Migliore and U. Nagel,
Jacobian Ideals of Hyperplane Arrangements and their Graded Betti Numbers.
arXiv:2508.12113.

\bibitem{MV}
S.E. Morey and R.H. Villarreal, Edge ideals: algebraic and combinatorial properties, in {\it Progress in commutative algebra 1}, 85--126.

\bibitem{OT} P. Orlik and H. Terao, \textit{Arrangements of hyperplanes}.
Grundlehren der Mathematischen Wissenschaften, 
\textbf{300}. Springer-Verlag, Berlin, 1992.

\bibitem{Sa}
K. Saito, 
Theory of logarithmic differential forms and logarithmic vector fields.
\textit{J. Fac. Sci. Univ. Tokyo} \textbf{27} (1980), 265--291.   


\bibitem{Stanley}
R.P. Stanley,
Combinatorics and commutative algebra, Second edition,
Progr. Math., vol. 41,
Birkh\"auser, Boston, 1996.


\bibitem{Terai}
N. Terai, Alexander duality theorem and Stanley-Reisner rings. S\=urikaisekikenky\=usho K\=oky\=uroku, No. 1078 (1999), 174--184.

\bibitem{T1}
H. Terao, 
Arrangements of hyperplanes and their freeness I, II. 
\textit{J. Fac. Sci. Univ. Tokyo} \textbf{27} (1980), 293--320.   

\bibitem{T2}
H. Terao, 
Generalized exponents of a free arrangement of hyperplanes and
Shephard-Todd-Brieskorn formula. \textit{Invent. math}. 
\textbf{63}  (1981),
159--179.

\bibitem{TranT}
T. N. Tran and S. Tsujie, 
MAT-free graphic arrangements and a characterization of strongly chordal graphs by edge-labeling.
\textit{Alg. Combin}. \textbf{6} (2023),  (6), 1447--1467.

\bibitem{Va}
A. Van Tuyl,
A Beginner's Guide to Edge and Cover Ideals,
Monomial Ideals, Computations and Applications,
63--94,
Lecture Notes in Math. \textbf{2083},
Springer, 2013.

\bibitem{Yuz}
S. Yuzvinsky, On generators of the module of logarithmic $1$-forms with poles along an arrangement. \textit{J. Algebraic Combin.} {\bf 4} (1995), no.~3, 253--269.

\bibitem{WY}
J. Watanabe and K. Yanagawa,
Vandermonde determinantal ideals.
Math. Scand. 125 (2019), no. 2, 179–184.
	
\end{thebibliography}
\end{document}